\documentclass{article}
\usepackage[utf8]{inputenc}

\usepackage{amsmath}
\usepackage{amssymb}
\usepackage{amsthm}
\usepackage{amsfonts}
\usepackage{graphicx}
\usepackage{multirow}
\usepackage{xcolor}
\usepackage{comment}
\usepackage{hyperref}
\usepackage{enumitem}
\usepackage{algorithm}
\usepackage{algpseudocode}
\usepackage[normalem]{ulem}

 \DeclareMathOperator*{\argmin}{arg\,min}

\newcommand\una[1]{{\color{blue}{{#1}}}}

\newcommand{\e}{\mathrm{E}}

\title{New M-estimator of the leading principal component}
\author{Joni Virta, Una Radoji\v{c}i\'c, Marko Voutilainen}
\newtheorem{theorem}{Theorem}
\newtheorem{lemma}{Lemma}
\newtheorem{corollary}{Corollary}

\usepackage{natbib}

\begin{document}
\maketitle

\begin{abstract}
    We study the minimization of the non-convex and non-differentiable objective function $v \mapsto \mathrm{E} ( \| X - v \| \| X + v \| - \| X \|^2 )$ in $\mathbb{R}^p$. In particular, we show that its minimizers recover the first principal component direction of elliptically symmetric $X$ under specific conditions. The stringency of these conditions is studied in various scenarios, including a diverging number of variables $p$. We establish the consistency and asymptotic normality of the sample minimizer. We propose a Weiszfeld-type algorithm for optimizing the objective and show that it is guaranteed to converge in a finite number of steps. The results are illustrated with two simulations.
\end{abstract}

\section{Introduction}\label{sec:introduction}

Let $P$ be a probability distribution in $\mathbb{R}^p$ that has expected value zero. In this work we study the minimization of the map $f_P: \mathbb{R}^p \mapsto \mathbb{R}$ defined as
\begin{align}\label{eq:main_concept}
    f_P(v) = \mathrm{E} \left( \| X - v \| \| X + v \| - \| X \|^2 \right),
\end{align}
where $X \sim P$. Estimators obtained as minimizers of objective functions such as $f_P$ are typically called $M$-estimators or argmax-estimators in the literature, see \cite{van2000asymptotic}. From the form of the objective function \eqref{eq:main_concept}, it is not immediately clear which aspect of the distribution $P$ it measures. Some intuition can be gained from the following two observations: (a) The objective function can also be written as $f_P(v) = \mathrm{E}( \| X X' - v v' \|_* - \| X X' \|_* )$, where $\| \cdot \|_*$ denotes the nuclear norm; see Lemma \ref{lem:nuclear_norm} in Appendix \ref{sec:proofs}. Consequently, if $P$ is strongly concentrated to a one-dimensional subspace, the minimizers of $f_P(v)$ should recover the direction of this subspace. (b) For a fixed vector $v \in \mathbb{R}^p$, the map $x \mapsto \| x - v \| \| x + v \|$ is on the $R$-radius sphere $\{ x \in \mathbb{R}^p \mid \| x \| = R \}$ minimized by $x = 
 \pm R v/\| v \|$, see Lemma \ref{lem:interpretation} in Appendix \ref{sec:proofs}. This again suggests that, if most of the variation in $P$ is concentrated along a single direction, then the minimizer of $f_P$ is likely to be found among vectors having this direction. Put in another way, both (a) and (b) imply that minimizing $f_P$ can be expected to find the leading principal component direction of the distribution~$P$.


As one of our main results we quantify the above heuristic and show that, given a minimizer $v_0$ of $f_P$, the corresponding unit-length vector $v_0/\| v_0 \|$ does, \textit{under specific conditions}, recover the direction of the leading principal component of $P$, whereas, the norm $\| v_0 \|$ serves as a measure of deviation of $P$ along this direction. Indeed, when $p = 1$, it is easy to check that $f_P$ is minimized by $\pm v_0 \in \mathbb{R}$ where $v_0$ is the median absolute deviation (MAD) of the univariate~$X$. Moreover, unlike regular PCA, which assumes finite second moments, the studied method requires fewer moments, depending on the result in question, see Sections~\ref{sec:population}--\ref{sec:sample}.

\begin{figure}
    \centering
    \includegraphics[width=1\textwidth]{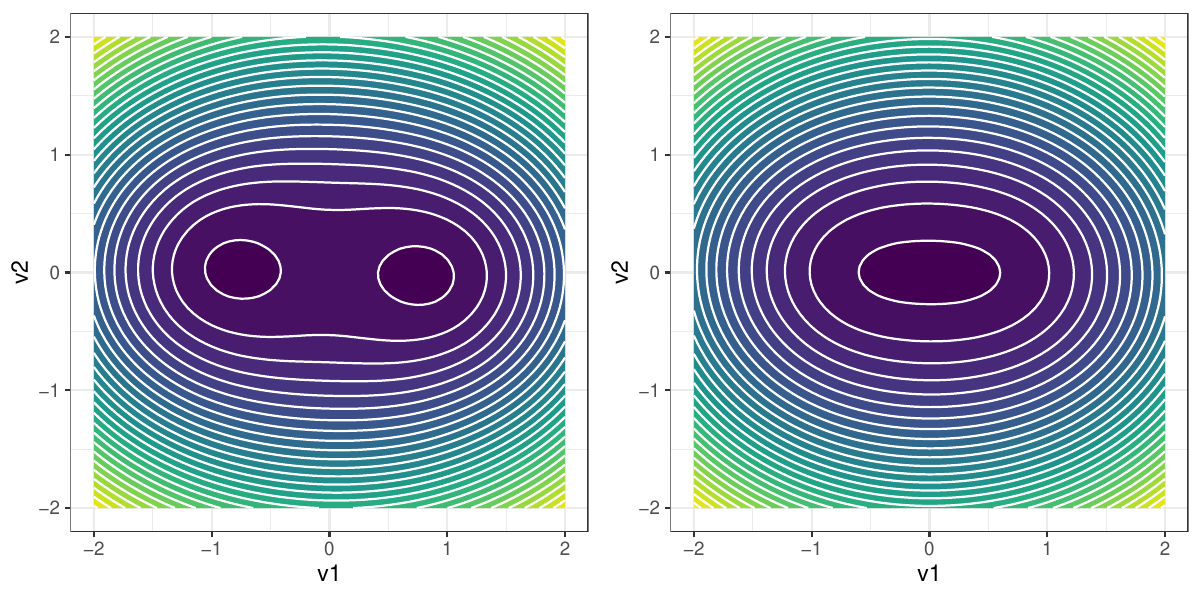}
    \caption{The contour plot of $f_P(v)$ as a function of $v$ based on a sample of size $n = 2000$. In the left panel, $X \sim \mathcal{N}_2(0, \mathrm{diag}(3, 1))$, whereas, in the right panel $X \sim \mathcal{N}_3(0, \mathrm{diag}(3, 1, 1))$ and we have restricted to the plane $v = (v_1, v_2, 0).$}
    \label{fig:intro}
\end{figure}

To illustrate the ``specific conditions'' mentioned above, consider Figure \ref{fig:intro}. In the left panel we show the contour plot of $f_P(v)$ as a function of $v$ based on a sample of size $n = 2000$ when $p = 2$ and $X \sim \mathcal{N}_2(0, \mathrm{diag}(3, 1))$. In this case, the direction of the first principal component is $(1, 0)$ and we observe that the two global minima of $f_P$ are indeed obtained roughly along two mirrored points on the $v_1$-axis. The right panel of Figure \ref{fig:intro} shows, analogously, the contour plot of $(v_1, v_2) \mapsto f_P(v_1, v_2, 0)$ when $X \sim \mathcal{N}_3(0, \mathrm{diag}(3, 1, 1))$. In this case, one would expect two minima to be again found on the $x$-axis, but, instead, it appears that the sole global minimizer lies at the origin. Our results below reveal that $f_P$ is unable to capture the principal component direction if the standard deviation of $X$ along this direction is too small. Moreover, we show that this threshold is increasing in the dimension $p$, see Theorem \ref{lemma:tau}, and grows at the rate $\sqrt{p}$. In the scenario of Figure~\ref{fig:intro}, where the standard deviation is $\sqrt{3} \approx 1.732$ in both cases, the exact thresholds are $1$ and $\approx 1.815$ for the cases $p = 2$ and $p = 3$, respectively, explaining the observed behavior.


The form of the objective function $f_P$ resembles that of the \textit{spatial median}, a classical robust measure of location of $P$, which is defined as any minimizer in $\mathbb{R}^p$ of the function
\begin{align*}
    \mu \mapsto \mathrm{E} \left( \| X - \mu \| - \| X \| \right),
\end{align*}
where $X \sim P$, see, e.g, \cite{milasevic1987uniqueness, mottonen2010asymptotic, paindaveine2021behavior} for works discussing the spatial median. However, despite this similarity, the two concepts turn out to measure wholly different aspects of $P$. Moreover, $f_P$ is, as a non-convex function, significantly more difficult to study. The nuclear norm representation, $f_P(v) = \mathrm{E}( \| X X' - v v' \|_* - \| X X' \|_* )$, also resembles the objective function of the \textit{median covariation matrix} \citep{cardot2017fast}, defined as any minimizer of $\Sigma \mapsto \mathrm{E}( \| X X' - \Sigma \|_F - \| X X' \|_F)$ over all positive definite matrices $\Sigma$, where $\| \cdot \|_F$ denotes the Frobenius norm. However, again, the two concepts are fundamentally different; our proposal estimates a singular mode of variation and a median covariation matrix, focusing on global dispersion.


Numerous approaches for conducting PCA are known in the literature, see the review in \cite{zhang2014novel}. However, our approach appears to be, despite the simple form of \eqref{eq:main_concept}, completely novel. We also note that the purpose of this work is \textit{not} to propose the most efficient method of computing PCA that beats all its competitors. Rather, our objective is to paint a comprehensive picture of the properties of the minimizers of the objective function $f_P$. 

The paper is organized as follows. In Section \ref{sec:population} we consider the minimization of $f_P$ on the population level, first for a general data distribution $P$ and then in the case of a specific elliptical family. In Section \ref{sec:sample}, we derive a limiting normal distribution for the minimizer of the sample version of the objective function under the same elliptical family, presenting also the equivalent results for classical PCA to allow an efficiency comparison. A Weiszfeld-type algorithm for minimizing $f_P$ in practice is given in Section \ref{sec:algorithm}, along with a proof of its convergence. In Section \ref{sec:simulation}, we present two simulation studies investigating the generalizability of the results and the efficiency of the estimator, respectively. Proofs of the technical results are collected to Appendices~\ref{sec:proofs}--\ref{sec:taylor}.


\section{Population-level behavior}\label{sec:population}

To define some notation, we let $\mathcal{O}^{p}$ be the set of all $p \times p$ orthogonal matrices and denote by $\mathcal{D}_1^{p} := \{ \sigma \mathrm{diag}(\lambda, 1, \ldots , 1) \mid \sigma > 0, \lambda > 1 \}$ the set of all $p \times p$ positive definite diagonal matrices with first diagonal elements strictly greater than the remaining $p - 1$, which are assumed equal to each other. The unit sphere in $\mathbb{R}^p$ is denoted by $\mathbb{S}^{p - 1}$.

\subsection{Results under general $P$}\label{sec:general_P}


Our first result shows that, despite the formulation of $f_P$ in \eqref{eq:main_concept} as a second-order quantity, $f_P$ is well-defined for any $P$ and requires no moment conditions to exist. This is guaranteed by the inclusion of the term $- \| X \|^2$ in \eqref{eq:main_concept}, which otherwise plays no role in the minimization of $f_P$.

\begin{theorem}\label{theo:moments}
    The quantity $f_P(v)$ is finite for all $v \in \mathbb{R}^p$.
\end{theorem}

The next result shows that the map $v \mapsto f_P(v)$ always has a minimizer.

\begin{theorem}\label{theo:existence}
    The map $v \mapsto f_P(v)$ admits a minimizer.
\end{theorem}

Differentiation reveals that $f_P$ is not convex, complicating its theoretical analysis. In particular, it does not seem possible to obtain any results regarding the uniqueness of minimizers in the case of general $P$. Due to this reason, we limit ourselves in the next section to a specific (wide) subset of elliptical distributions, under which several natural results can be derived. However, we first observe that, even for general $P$, the minimizer obeys a natural equivariance property under orthogonal transformations.

\begin{theorem}\label{theo:equivariance}
    Let $v_0$ be a minimizer of $f_P$. Then, for any $O \in \mathcal{O}^p$, the vector $O v_0$ is a minimizer of $f_{P_O}$, where $P_O$ denotes the distribution of $O X$ with $X \sim~P$.
\end{theorem}

The proof of Theorem \ref{theo:equivariance} is straightforward and omitted.

\subsection{Results under ellipticity}

Fixing now $O =  (o_1, \ldots, o_p) \in \mathcal{O}^{p}$, $\Lambda \in \mathcal{D}_1^{p}$, we say that a $p$-variate distribution $P$ belongs to the family $\mathcal{E}(O, \Lambda)$ if the random vector $X \sim P$ admits the representation $X = O \Lambda Z$ for some random $p$-vector $Z$ such that
\begin{enumerate}
    \item $Z$ has a spherical distribution,
    \item $R := \| Z \|$ admits a density $f_R(r): (0, \infty) \to (0, \infty) $ w.r.t. the Lebesgue measure,
    \item $r \mapsto f_R(r)/r^{p - 1}$ is strictly decreasing, and
    \item $\mathrm{E}(R)$ exists as finite.
\end{enumerate}
By ``spherical distribution'' in Item 1, we mean that $Z$ and $VZ$ have the same distribution for every $V \in \mathcal{O}^{p}$, see \cite{fang1990symmetric}. In Item 4, the finite first moment is required later on for the existence of the gradient of $f_P$. Note that stating that $P \in \mathcal{E}(O, \Lambda)$ does not uniquely fix the distribution $P$ as the notation leaves the distribution of $\| Z \|$ unspecified. Our next result is stated in terms of the \textit{spatial sign covariance} matrix $S_P := \mathrm{E}(XX'/\| X \|^2)$ and its leading eigenvalue $\phi_1(S_P)$. Note that $\mathrm{tr}(S_P) = 1$, meaning that $\phi_1(S_P) \in [1/p, 1]$. 

\begin{theorem}\label{theo:main_1}
    Let $X \sim P \in \mathcal{E}(O, \Lambda)$.
    \begin{itemize}
        \item[(i)] If $\phi_1(S_P) > 1/2$, then there exists $\psi > 0$ such that the minimizers of $f_P$ are precisely the vectors $v = \pm \psi o_1$.
        \item[(ii)] If $\phi_1(S_P) \leq 1/2$, then $f_P$ is uniquely minimized at $v = 0$. 
    \end{itemize}
\end{theorem}

The proof of Theorem \ref{theo:main_1}, given in Appendix \ref{sec:proofs}, is based on treating the magnitude $\| v \|$ and the direction $v / \| v \|$ of the minimizer separately. For the magnitude, we show that the objective function admits a convex representation on every ray emanating from the origin, and that whether the ray contains a non-trivial minimizing value depends on the eigenvalue of the matrix $S_P$. For the direction, we establish a general conditional stochastic dominance result for objective functions depending on the data through $\| X \|^2$ and a projection $(X'u)^2$ onto a single direction, using which we prove that the choices $u = \pm o_1$ dominate all others under the model.  

Theorem \ref{theo:main_1} reveals that whether the minimizer of $f_P$ allows identifying the leading principal component direction depends on the magnitude of the largest eigenvalue of the spatial sign covariance matrix $S_P$. Under a general elliptical model, the eigenvalues of $S_P$ are known to admit closed-form expressions only for $p = 2$, see Proposition 1 in \cite{durre2015spatial}. However, under the more restrictive model $\mathcal{E}(O, \Lambda)$, where the final $p - 1$ elements of $\Lambda$ are equal, something more can be said. From Proposition 1 in \cite{durre2016eigenvalues} it follows that, under our assumptions, the eigenvalues of the spatial sign covariance matrix satisfy $\phi_1(S_P) > \phi_2(S_P) = \cdots = \phi_p(S_P)$. Moreover, since the trace of $S_P$ is unity, its spectrum is thus
\begin{align*}
    \tau, \frac{1 - \tau}{p - 1}, \ldots, \frac{1 - \tau}{p - 1},
\end{align*}
for some $\tau = \tau(\lambda, p) := \phi_1(S_P) \in (0, 1)$. From Proposition 3 in \cite{durre2016eigenvalues} we obtain the following integral representation for $\tau$.
\begin{lemma}\label{lem:integral}
    Let $X \sim P \in \mathcal{E}(O, \Lambda)$. Then,
    \begin{align}\label{eq:durre_integral}
    \tau(\lambda, p) = \frac{\lambda^2}{2} \int_0^\infty (1 + \lambda^2 x)^{-3/2} (1 + x)^{(1 - p)/2} dx.
\end{align}
\end{lemma}
In Appendix \ref{sec:integration} we compute the integral in dimensions $p\in\{2,3,4\}$, showing that it leads to increasingly complex expressions. Consequently, the analytical solving of the inequality $\tau(\lambda , p) > 1/2$ is not feasible. However, some qualitative, asymptotic claims can still be made. Denote next by $\tau_0(\lambda, p) := \{ 1 - \tau(\lambda, p) \}/(p - 1)$ the magnitude of the tail eigenvalue of $S_P$ and by  $\lambda_p^*$ the unique root of $\lambda \mapsto \tau(\lambda, p) - 1/2$ for fixed $p$. 

\begin{theorem}
\label{lemma:tau}
Let $X \sim P \in \mathcal{E}(O, \Lambda)$. Then,
\begin{enumerate}
    \item[(1)] For fixed $\lambda > 1$, the function $p \mapsto \tau(\lambda, p)$ is strictly decreasing and satisfies $\tau(\lambda, p) = \lambda^2/p + o(1/p)$ as $p \rightarrow \infty$.
    \item[(2)] For fixed $\lambda > 1$ the function $p \mapsto \tau_0(\lambda, p)$ is strictly decreasing and satisfies $\tau_0(\lambda, p) = 1/p + o(1/p)$ as $p \rightarrow \infty$. 
    \item[(3)] For fixed $p$, the function $\lambda \mapsto \tau(\lambda, p)$ is strictly increasing and satisfies $\tau(\lambda, p) = 1 + o(1) $ as $\lambda \rightarrow \infty$.
    \item[(4)] The function $p \mapsto \lambda_p^*$ is strictly increasing and, for all $\varepsilon > 0$, satisfies $p^{\frac{1}{2} - \varepsilon} < \lambda_p^* < p^{\frac{1}{2}+\varepsilon}$ for all $p$ large enough. 
    \item[(5)] There exists a unique $C > 0$ such that $\tau(Cp^\frac12, p) \rightarrow 1/2$ as $p \rightarrow \infty$, whose approximate value is $C \approx 1.634$.
\end{enumerate}
\end{theorem}

In particular, part (4) of Theorem \ref{lemma:tau} reveals that satisfying the condition $\phi_1(S_P) = \tau(\lambda, p) > 1/2$ gets increasingly more difficult as $p$ grows. Moreover, parts (4) and (5) together state that, for large $p$, $f_P$ is able to identify the first PC if the standard deviation parameter $\lambda$ related to the first principal component is approximately at least $1.634 p^{1/2}$.




Another question of interest is whether the norm $\psi$ of the minimizer in Theorem \ref{theo:main_1} can be further quantified. As we remarked in Section \ref{sec:introduction}, when $p = 1$, the absolute value of the minimizer is the median absolute deviation of $X$. Hence, $\psi$ can be seen as an analogue of MAD for multivariate data. It is straightforwardly seen from the definition of $f_P$ that, if the random vector $X$ is scaled as $X \mapsto bX$ for some $b \in \mathbb{R}$, then the norm of the minimizer transforms as $\psi \mapsto |b| \psi$, making it a \textit{measure of scale}. However, obtaining a closed-form solution for $\psi$ appears to be infeasible, as evidenced by the integral representation we derive in Appendix~\ref{sec:norm_minimizer}. Instead, we next present a result on the asymptotic behavior of the norm $\psi = \psi(\lambda,p)$ of the minimizer when either $p$ or $\lambda$ is let to diverge to infinity.

\begin{theorem}
    \label{thm:normofminimizer}
    Let $X \sim P \in \mathcal{E}(O, \Lambda)$. Then,
    \begin{enumerate}
        \item[(1)] For every fixed $\lambda > 1$, it holds that $\psi(\lambda,p) = 0$ for all $p$ large enough.
        \item[(2)] For fixed $p$, we have for all $\varepsilon > 0$ that $\lambda^{1 - \varepsilon} < \psi(\lambda,p) < \lambda^{1 + \varepsilon}$ for all $\lambda$ large enough. 
    \end{enumerate}
\end{theorem}

Part (1) of Theorem \ref{thm:normofminimizer} essentially states that, when the dimension $p$ of the data grows, any fixed signal strength $\lambda$ gets drowned by the noise contained in the final $p - 1$ dimensions, making the first PC non-identifiable. Part (2) states that, for fixed tail eigenvalues, the norm grows proportional to the standard deviation parameter $\lambda$ of the first principal component. This result can thus be seen as an asymptotic version of the fact that $\psi$ is a measure of scale.

In the next section, we study the norm of the minimizer, along with the minimizing direction, from the viewpoint of large-sample asymptotics.

\section{Large-sample behavior under ellipticity}\label{sec:sample}

As our next task, we study the minimizers of the sample version of the objective function. Namely, let $P_n$ be the empirical distribution of a random sample $X_1, \ldots, X_n$ drawn from $P \sim \mathcal{E}(O, \Lambda)$. Then, we are interested in the minimizers of $f_{P_n}$, defined as,
\begin{align}\label{eq:sample_objective_function}
    f_{P_n}(v) = \frac{1}{n} \sum_{i = 1}^n \left( \| X_i - v \| \| X_i + v \| - \| X_i \|^2 \right).
\end{align}

The asymptotic theory of argmax estimators is greatly simpler if either (i) the objective function is convex, or (ii) the parameter space is a compact set. Our current scenario satisfies neither of these conditions, and we solve this problem by carrying out the optimization of \eqref{eq:sample_objective_function} in a very specific sequence of compact sets. For $R > 0$, we denote $\mathcal{B}(R) := \{ x \in \mathbb{R}^p \mid \| x \| \leq R \}$ and let $v_n$ denote any minimizer of \eqref{eq:sample_objective_function} over $\mathcal{B}(H_n + 1/2)$, where the data-dependent quantity $H_n$ is defined in Appendix \ref{sec:radius}. The significance behind this choice is that all minimizers of $f_{P_n}$ can be shown to reside in $\mathcal{B}(H_n + 1/2)$, letting us restrict our attention to this sequence of domains. Recall in the sequel that $\psi o_1$ denotes a minimizer of the population-level objective function $f_P$, in the notation of Theorem \ref{theo:main_1}.

\begin{theorem}\label{thm:thm7}
    Let $X \sim P \in \mathcal{E}(O, \Lambda)$ and assume that $\phi_1(S_p) > 1/2$. Then there exists a sequence of signs $s_n \in \{-1, 1\}$ such that $s_n v_n \rightarrow_P \psi o_1$, as $n \rightarrow \infty$, where $v_n$ is as above.
\end{theorem}

Having established consistency, we next move to asymptotic normality. A key component of the proof is Appendix \ref{sec:taylor}, where we use the dominated convergence theorem to derive a second-order Taylor expansion of the population objective function. In the following result $e_1$ denotes the first standard basis vector.


\begin{theorem}\label{thm:thm8}
    Let $X \sim P \in \mathcal{E}(O, \Lambda)$ and assume that the density $f_R(r)$ of $R := \| Z \|$ is bounded, $\phi_1(S_p) > 1/2$ and $\mathrm{E}( Z_1^2 )<\infty$. Then there exists a sequence of signs $s_n \in \{-1, 1\}$ such that
    \begin{align*}
        \sqrt{n} (s_n v_n - \psi o_1) \rightsquigarrow \mathcal{N}_p(0, \Sigma),
    \end{align*}
     as $n \rightarrow \infty$, where $\Sigma=O \{ q_1 e_1 e_1' + q_2 (I_p - e_1 e_1') \} O'$ and $q_1, q_2$, are given on the last line of the proof of this result.
\end{theorem}


As the direction $v_n/\| v_n \|$ and magnitude $\| v_n \|$ of the minimizer both describe different aspects of the distribution $P$, we next provide an auxiliary result that allows separating the limiting distribution in Theorem \ref{thm:thm8} to two parts, giving the joint distribution of the two components in terms of the limiting distribution of~$v_n$.

\begin{lemma}\label{lem:decomposition}
    Let $h_n$ be a sequence of random $p$-vectors that satisfies $\sqrt{n} (h_n - h) \rightsquigarrow \mathcal{N}_p(0, \Sigma)$, as $n \rightarrow \infty$, for some non-zero $h \in \mathbb{R}^p$ and positive semi-definite $\Sigma \in \mathbb{R}^{p \times p}$. Then,
    \begin{align*}
        \sqrt{n} \begin{pmatrix}
            \| h_n \| - \| h \| \\
            \frac{h_n}{\| h_n \|} - \frac{h}{\| h \|}
        \end{pmatrix} 
        \rightsquigarrow \mathcal{N}_{p + 1} \left( 0, \frac{1}{\| h \|^2} \begin{pmatrix}
            h' \Sigma h & h' \Sigma Q \\
            Q \Sigma h & Q \Sigma Q
        \end{pmatrix} \right),
    \end{align*}
    where $Q := I_p - hh'/\| h \|^2$. 
\end{lemma}

As a corollary of Theorem~\ref{thm:thm8} and Lemma~\ref{lem:decomposition} we get the limiting distributions of the unit norm estimator $v_n/\|v_n\|$ of the leading principal direction $o_1$, as well as the estimator $\| v_n \|$ of the measure of scale $\psi$. Interestingly, these turn out to be asymptotically independent. The proof follows straightforwardly from the previous two results and is thus omitted.

\begin{corollary}\label{cor:asymp_split}
    Let $X \sim P \in \mathcal{E}(O, \Lambda)$ and assume that the density $f_R(r)$ of $R := \| Z \|$ is bounded, $\phi_1(S_p) > 1/2$ and $\mathrm{E}( Z_1^2 )<\infty$. Then there exists a sequence of signs $s_n \in \{-1, 1\}$ such that
    \begin{align*}
        \sqrt{n} \begin{pmatrix}
            \|v_n\|-\psi\\
            s_n\frac{v_n}{\|v_n\|} -  o_1      \end{pmatrix}\rightsquigarrow \mathcal{N}_{p + 1} \left( 0, \begin{pmatrix}
            q_1 & 0' \\
            0 & \frac{q_2}{\psi^2} O\left(I_p-e_1e_1'\right)O'
        \end{pmatrix} \right),
    \end{align*}
    as $n \rightarrow \infty$, where $q_1, q_2$ are as in Theorem \ref{thm:thm8}.
\end{corollary}

Finally, we still provide, for comparison, the counterpart of Theorem \ref{thm:thm8} for the standard PCA estimator. Let $u_n$ denote the leading unit-length eigenvector of the sample covariance matrix computed from the empirical distribution $P_n$. The following result is well-known, but for completeness, we provide a proof. Note that the result requires finite moments of order 4, whereas finite moments of order 2 are sufficient in Theorem \ref{thm:thm8}.

\begin{theorem}\label{theo:pca}
    Let $X \sim P \in \mathcal{E}(O, \Lambda)$ and assume that $\mathrm{E} (Z_1^4) < \infty$. Then there exists a sequence of signs $s_n \in \{-1, 1\}$ such that
    \begin{align*}
        \sqrt{n} (s_n u_n - o_1) \rightsquigarrow \mathcal{N}_p \left( 0, \frac{\lambda^2}{(\lambda^2 - 1)^2} \frac{\mathrm{E}(Z_1^2 Z_2^2)}{\{ \mathrm{E}(Z_1^2) \}^2} O (I_p - e_1 e_1') O' \right),
    \end{align*}
     as $n \rightarrow \infty$.
\end{theorem}

Thus, under the respective assumptions, both $s_n v_n/\| v_n \|$ and $u_n$ estimate the same vector, namely the direction of the leading principal component of $X$. We conduct the comparison between their limiting covariance matrices with simulations in Section \ref{sec:simulation}, since the expressions of $q_1, q_2$ given in the proof of Theorem \ref{thm:thm8} in Appendix \ref{sec:proofs} depend on quantities such as $\mathrm{E}\{ \| X + \psi o_1 \|/\| X - \psi o_1 \| \}$ that do not admit closed-form expressions.

\section{Computation}\label{sec:algorithm}

Given a fixed sample $X_1, \ldots, X_n \in \mathbb{R}^p$, we propose minimizing the objective function $f_{P_n}$ in \eqref{eq:sample_objective_function} using a modified Weiszfeld-type algorithm, see \cite{beck2015weiszfeld}, producing a sequence of solution candidates $v_1, v_2, \ldots$. As it plays no role in the minimization, we ignore the additional term $-\| X_i \|^2$ in \eqref{eq:sample_objective_function} altogether in this section. Since the objective function is non-differentiable at the points $v = \pm X_i$, $i = 1, \ldots, n$, we derive separate update rules depending on whether the current iterate $v_k$ equals one of the sample points (up to sign), or not.


Assume first that $v_k \neq \pm X_i$, $i = 1, \ldots, n$. Denoting now 
$$
L(v)=\begin{cases}\displaystyle
    \frac{1}{n} \sum_{i = 1}^n  \frac{ \| v + X_i \|^2+\| v - X_i \|^2 } { \| v - X_i \|\| v + X_i \| },\mbox{ for } v\neq \pm X_j,\,j=1,\dots,n,\\
    \displaystyle\frac{1}{n} \sum_{i \neq j}  \frac{ \| v + X_i \|^2+\| v - X_i \|^2 } { \| v - X_i \|\| v + X_i \| },\mbox{ for } v=\pm X_j,\,j=1,\dots,n, 
\end{cases}  
$$
we propose to update the current solution candidate $v_k$ with
\begin{align}\label{eq:T_iterative}
v_{k+1}=T(v_k) := \frac{1}{L(v_{k})}\frac{1}{n} \sum_{i = 1}^n  \frac{ \| v_k + X_i \|^2-\| v_k - X_i \|^2 } { \| v_k - X_i \|\| v_k + X_i \| } X_i.
\end{align}
We observe that this is, in fact, a gradient descent approach with a step size 
$L(v_k)^{-1}$. To see this, Lemma \ref{lemma:kuhn} in Appendix \ref{sec:proofs} shows that, for $v\neq \pm X_i$, $i=1,\dots,n$, the gradient of the objective function has the form
\begin{align*}
    \nabla f_{P_n}(v) = \frac{1}{n} \sum_{i = 1}^n \left\{ \frac{ \| v + X_i \| } { \| v - X_i \| } ( v - X_i ) + \frac{ \| v - X_i \| } { \| v + X_i \| } ( v + X_i ) \right\}.
\end{align*}
Hence, we can write
\begin{align*}
    T(v) & = -\frac{1}{nL(v)} \sum_{i = 1}^n  \frac{ \| v + X_i \| } { \| v - X_i \| } (-X_i) -  \frac{1}{nL(v)} \sum_{i = 1}^n  \frac{\| v - X_i \| } { \| v + X_i \| } X_i \\
   & = \frac{-1}{nL(v)}\left( \sum_{i = 1}^n  \frac{ \| v + X_i \| } { \| v - X_i \| } (v-X_i) +  \sum_{i = 1}^n  \frac{\| v - X_i \| } { \| v + X_i \| } (v+X_i)\right)\\
   &~~~~ + v\frac{1}{L(v)}\frac{1}{n}\sum_{i=1}^n \frac{ \| v + X_i \|^2+\| v - X_i \|^2 } { \| v - X_i \|\| v + X_i \| }
   \\
   & = v - \frac{1}{L(v)}\nabla f_{P_n}(v),
\end{align*}
establishing the gradient descent form. The following lemma shows that the step size $L(v)^{-1}$ is chosen such that the iterations decrease the value of the objective $f_{P_n}$, provided the estimates do not, up to sign, coincide with the sample points $X_i$, $i=1,\dots,n$.

\begin{lemma}\label{alg::lemma1}
    For $v\neq \pm X_i$, $i=1,\dots,n$, let $T(v)$ be as defined in \eqref{eq:T_iterative}. Then, $f_{P_n}(T(v))\leq f_{P_n}(v)$, where equality holds if and only if $T(v) = v$.
\end{lemma}

In case the current solution candidate $v_k$ is, up to sign, equal to one of the sample points, Lemma \ref{lemma:kuhn} in Appendix \ref{sec:proofs} shows that minimizers should be searched among the null points of a \textit{modified gradient}; see Lemma \ref{lemma:kuhn} for more details. Let $M$ denote the maximal non-zero value in the set $\{ 2 \| X_1 \|, \ldots, 2 \| X_n \| \}$. Then, we update $v_k$ as
\begin{equation}\label{eq:modified_T}
v_{k+1}=T_i(\pm X_i) :=\begin{cases}
    \pm X_i,\quad &\text{if}\quad \|\nabla f_i\|-2\|X_i\| \leq 0\\
    \pm X_i - \varepsilon d_i,\quad &\text{if}\quad \|\nabla f_i\|-2\|X_i\| > 0, 
\end{cases}
\end{equation}
where $0<\varepsilon< M$ is chosen such that $f_{P_n}(T_i(\pm X_i))<f_{P_n}(\pm X_i)$, and $d_i := -\nabla f_i/\|\nabla f_i\|$ is a unit vector in the direction of $-\nabla f_i$, see Lemma \ref{lemma:kuhn} for its definition. The following lemma shows that there indeed exists $\varepsilon$ such that the mapping $\pm X_i \mapsto \pm X_i-\varepsilon d_i$ lowers the objective.

\begin{lemma}\label{alg::lemma2}
    For $v = \pm X_i$, $i = 1,\dots,n$, let $T_i(v)$ be as defined in \eqref{eq:modified_T}. Then there exists $\varepsilon \in (0, M)$ such that $f_{P_n}(T_i(v))\leq f_{P_n}(v)$, where equality holds if and only if $T_i(v) = v$.
\end{lemma}

Lemmas \ref{lem:Tv_is_v_1} and \ref{lem:Tv_is_v_2} in Appendix Z show that, for all $v \neq 0$, we have $T(v) \neq -v$. This explains why in Lemmas \ref{alg::lemma1} and \ref{alg::lemma2} a necessary condition for the equality is $T(v) = v$, even though the objective function is symmetric in the sense that $f_{P_n}(-v) = f_{P_n}(v)$. We note also that the upper bound $\varepsilon < M$ in Lemma \ref{alg::lemma2} is set precisely for this reason, see Lemma \ref{lem:Tv_is_v_2} for more details.\\

Based on Lemmas \ref{alg::lemma1} and \ref{alg::lemma2}, given a fixed sample $X_1, \ldots, X_n \in \mathbb{R}^p$, we propose  minimizing the objective function $f_{P_n}$ in \eqref{eq:sample_objective_function} using Algorithm~\ref{alg}. Observe that the iterative scheme defined by \eqref{eq:T_iterative} and \eqref{eq:modified_T} depends on an initial approximation $v_0$. We choose the \textit{informative} initial approximation defined as the leading eigenvector $u_0$ of the sample covariance matrix of the random sample $X_1,\dots X_n$. As $u_0$ is a unit length vector, one could additionally tune its length by optimizing $f_{P_n}$ in the direction of $u_0$. For this purpose, for unit length $u\neq \pm X_i/\|X_i\|$, $i=1,\dots,n$, define $g_{P_n}(\lambda;u) := f_{P_n}(\lambda u)$. Lemma~\ref{lem:convexity_of_h}, along with the fact that sum of convex functions is a convex function, gives that for every fixed unit length $u\neq \pm X_i/\|X_i\|$, $i=1,\dots,n$, the map $\lambda\mapsto g_{P_n}(\lambda^{1/2};u)$ is strictly convex, making its optimization straightforward via standard gradient descent.


\begin{algorithm}
\caption{Iterative minimization of \eqref{eq:sample_objective_function}}\label{alg}
\begin{algorithmic}[1]
\Require Sample $X_1,\dots,X_n$, tolerance $\Delta > 0$; 
\State $v_{\mathrm{old}}$ $\gets$ leading eigenvector of $\displaystyle\frac{1}{n}\sum_{i=1}^n(X_i-\bar{X})(X_i-\bar{X})'$, for $\displaystyle\bar{X}=\frac{1}{n}\sum_{i=1}^n X_i$;
\State $err\gets \Delta+1$;
\While{$err > \Delta$}
\If{$v_{\mathrm{old}}\neq \pm X_i$, $i=1,\dots,n$}
    \State $v_{\mathrm{new}}\gets T(v_{\mathrm{old}})$;
    \State $err\gets f_{P_n}(v_{\mathrm{old}})-f_{P_n}(v_{\mathrm{new}})$;
\ElsIf{$v_{\mathrm{old}} = \pm X_i$, for some $i = 1,\dots,n$}
     \State $v_{\mathrm{new}}\gets T_i(v_{\mathrm{old}})$;
     \If{$\|\nabla f_i\|-2\|X_i\| \leq 0$}
    \State $err\gets f_{P_n}(v_{\mathrm{old}})-f_{P_n}(v_{\mathrm{new}})$;
    \ElsIf{$\|\nabla f_i\|-2\|X_i\| > 0$} 
    \State $err\gets \Delta+1$; 
    \EndIf
\EndIf 
\State $v_{\mathrm{old}} \gets v_{\mathrm{new}}$;
\EndWhile
\If{$v_{\mathrm{new}} \neq 0$}
\State $\lambda \gets \|v_{\mathrm{new}}\|$;
\If{$v_{\mathrm{new}}\neq \pm X_i$, $i=1,\dots,n$}
    \State $\lambda \gets\argmin_{\lambda>0}g_{P_n}(\lambda;v_{\mathrm{new}}/\|v_{\mathrm{new}}\|)$;
\EndIf
\State \textbf{return} $\lambda v_{\mathrm{new}}/\|v_{\mathrm{new}}\|$;
\ElsIf{$v_{\mathrm{new}} = 0$}
\State \textbf{return} $v_{\mathrm{new}}$;
\EndIf
\end{algorithmic}
\end{algorithm}

The artificial increase of the error in \texttt{Step 12} of the Algorithm is to ensure that the convergence has not been reached simply due to the arbitrarily small update required to ``jump away'' from a non-optimal sample point $X_i$ to a non-sample point, or to a not-yet-explored sample point; see the proof of Corollary \ref{cor:alg} for more insight. We emphasize, however, that in practice, Algorithm~\ref{alg} does not reach the sample points, and thus \texttt{Steps 7 - 13} are, in some sense, obsolete. A similar observation was made in \cite{beck2015weiszfeld} for the modified Weiszfeld’s method for finding the spatial median. For the same reason we have also not included the selection of $\varepsilon$ in Lemma \ref{alg::lemma2} to Algorithm \ref{alg}. However, if wanted, its proper value could easily be found with a line search; see the proof of Lemma \ref{alg::lemma2}.

Algorithm \ref{alg} could be modified by additionally optimizing for the length of the current approximation $v_k$ in \texttt{Step 5}. However, numerical experiments showed that this strategy does not improve the performance of the algorithm, and we thus avoid it. The implementation of Algorithm \ref{alg} is available in \url{https://github.com/jmvirta/spatialPCA/}.

Finally, the following result guarantees that the main iterative part of Algorithm \ref{alg} (the \texttt{while}-loop) converges within a finite number of steps.

\begin{corollary}\label{cor:alg}
    For every tolerance $\Delta>0$, Algorithm \ref{alg} converges after a finite number of iterations.
\end{corollary}

\section{Simulations}\label{sec:simulation}

\subsection{Simulation study \#1}

The purpose of this simulation is both to demonstrate Theorem \ref{theo:main_1} in practice and to investigate its robustness to our assumptions. We generated three-variate data with independent margins that admit the covariance matrix $\mathrm{diag}(\theta^2, 1, 1)$, $\theta \geq 1$, from three models. In the first model, the marginal distributions were taken to be normal, in the second model uniform on intervals, and in the third symmetric scaled Bernoulli variates. Hence, the first model falls within the scope of Theorem \ref{theo:main_1}, whereas the remaining two do not.

We generated $n = 100, 200, 400, 800$ observations from the models and considered a regular grid of 20 standard deviation values from $\theta = 1$ to $\theta = 3$. Each combination of settings was repeated 1000 times and in each case we estimated $v_n$, using Algorithm \ref{alg}. We then computed the values of $| v_{n1} |$ and $\sqrt{v_{n2}^2 + v_{n3}^2}$, the averages of which over the replicates are shown in Figure \ref{fig:simu_1}. The vertical lines in the plot correspond to the approximate value $1.815$, computed numerically using the function $I_3(\lambda)$ in Appendix \ref{sec:integration}, which is the threshold beyond which $\phi(S_P) > 1/2$ and the first principal direction is identifiable under elliptical data, see Theorem~\ref{theo:main_1}. 

\begin{figure}
    \centering
    \includegraphics[width=1.0\linewidth]{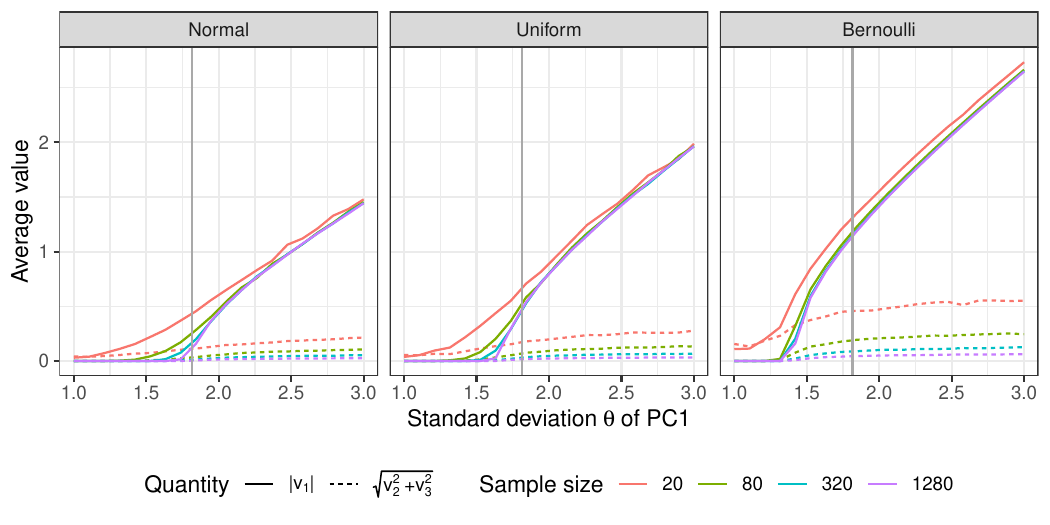}
    \caption{The average values of $| v_{n1} |$ and $\sqrt{v_{n2}^2 + v_{n3}^2}$ over 1000 replicates for different combinations of sample size $n$, model and $\theta$ when $\mathrm{Cov}(X) = \mathrm{diag}(\theta^2, 1, 1)$. The grey vertical line indicates the identifiability threshold predicted by Theorem \ref{theo:main_1} for elliptical distributions.}
    \label{fig:simu_1}
\end{figure}

From Figure \ref{fig:simu_1} we observe the following: (i) Under normality, the results confirm the theoretical claims in Theorem \ref{theo:main_1}: for large $n$ there exists a cut-off point close to the value $\theta = 1.815$ such that, for $\theta$ smaller than this, $| v_{n1} |$ converges to zero and, for $\theta$ larger than this value, $| v_{n1} |$ converges to a positive constant. Moreoever, $\sqrt{v_{n2}^2 + v_{n3}^2}$ appears to converge to zero for all values of $\theta$. (ii) The analogous holds true also for the uniform and binary models, but the threshold for the behavior is smaller than under normality. E.g., for Bernoulli data, having approximately $\theta > 1.3$ appears sufficient to identify the direction of the first principal component. This effect is likely caused by the more degenerate nature of the data under uniform and Bernoulli models, which also makes the method more sensitive to small sample sizes, as evidenced by the slower convergence of the lines corresponding to $\sqrt{v_{n2}^2 + v_{n3}^2}$ under these models.

\begin{figure}
    \centering
    \includegraphics[width=1.0\linewidth]{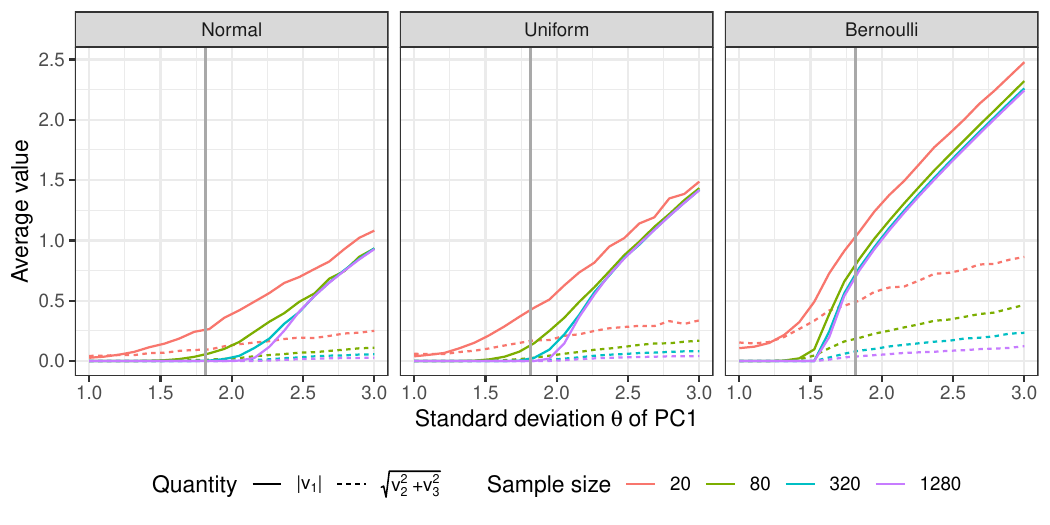}
    \caption{The average values of $| v_{n1} |$ and $\sqrt{v_{n2}^2 + v_{n3}^2}$ over 1000 replicates for different combinations of sample size $n$, model and $\theta$ when $\mathrm{Cov}(X) = \mathrm{diag}(\theta^2, \theta, 1)$. The grey vertical line indicates the identifiability threshold predicted by Theorem \ref{theo:main_1} for elliptical distributions.}
    \label{fig:simu_2}
\end{figure}

Next, we changed the covariance matrix of the data to $\mathrm{diag}(\theta^2, \theta, 1)$, keeping otherwise the same specifications, and reran the simulation. The purpose here is thus to investigate whether the results continue to hold when the final $p - 1$ eigenvalues are not identical. The results are shown in Figure~\ref{fig:simu_2} revealing that this is indeed the case. However, having unequal tail variances appears to increase the identifiability threshold. E.g., for normal data, a standard deviation of roughly $2.1$ is required for the minimizer to attain the direction of PC1. This is not unexpected since increasing the tail variability makes it more difficult to detect the direction of largest variation, requiring it to stand out more.


All in all, based on the simulations, we conclude that the behavior predicted in Theorem \ref{theo:main_1} is not solely confined to the elliptical model $\mathcal{E}(O, \Lambda)$ but continues to hold also under various other structures. 

\subsection{Simulation study \#2}

Next, we compare the limiting efficiencies of our proposed estimator $s_n v_n/\|v_n\|$ and the PCA-estimator $u_n$. We do this by numerically evaluating the expected values required to compute the limiting covariance matrices of the two estimators in Corollary \ref{cor:asymp_split} and Theorem \ref{theo:pca}. As our setting, we take the model $\mathcal{E}(O, \Lambda)$ with $Z$ having a standard multivariate $t_{\nu}$-distribution with dimensionality $p = 3, 5, 10, 50, 100$ and degrees of freedom $\nu = 3,5,10$. 
We consider a range of values $\lambda = 2,\dots,40$ and, without loss of generality, take $O = I_p$ and $\sigma^2 = 1$. The quantities are estimated based on $1000$ replications of samples of size $n = 500,1000,10000$ and the logarithmic spectral norms of the estimated asymptotic covariance matrices (ASCOV) are plotted as a function of $\lambda$ in Figure \ref{fig:fig4}. Thus, lower values indicate a more efficient estimator. We use the PCA-solution as an initial value for Algorithm \ref{alg}, meaning that the experiment essentially investigates whether our proposal can improve upon or ``refine'' this starting value. 

\begin{figure}[H]
    \centering
    \includegraphics[width=0.9\linewidth]{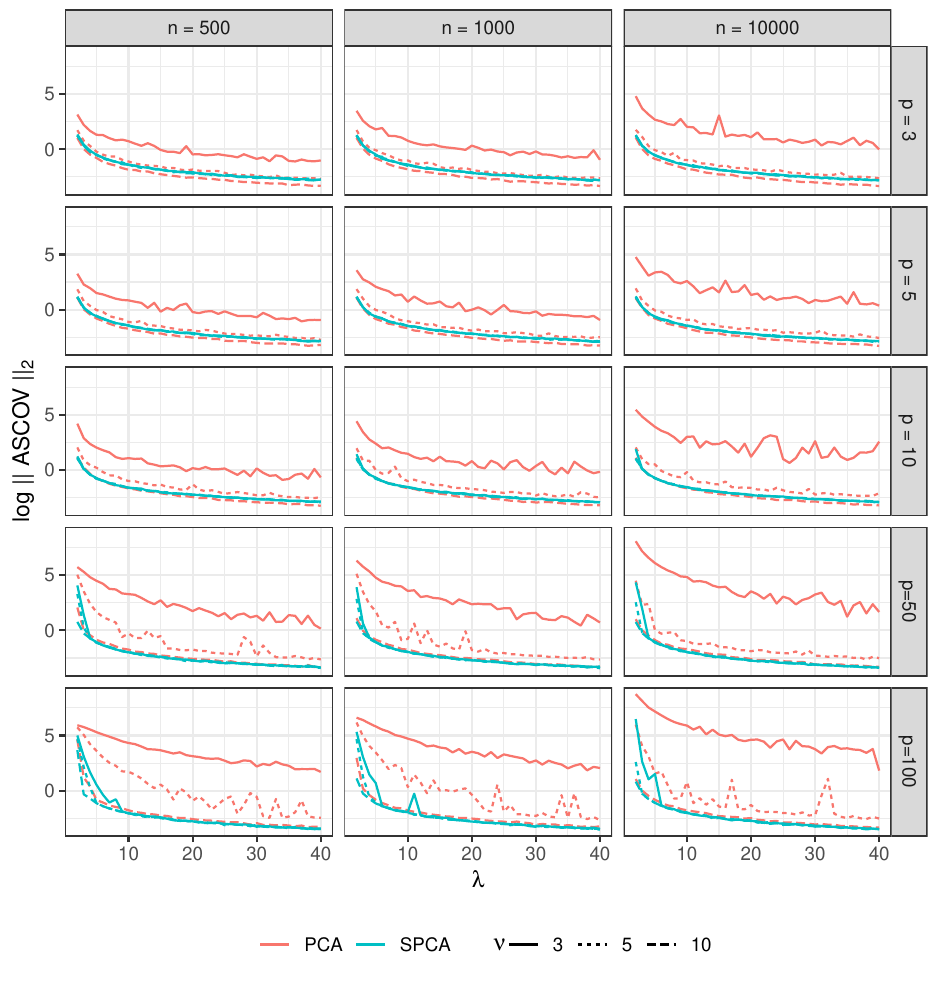}
    \caption{The logarithmic spectral norm of the estimated asymptotic covariance $ASCOV$ of $\sqrt{n}(s_n\tfrac{w_n}{\|w_n\|}-o_1)$, where $w_n/\|w_n\|$ is the 
    unit norm estimator of the leading principal direction based on the proposed approach (SPCA, blue) and PCA (red).  
    }
    \label{fig:fig4}
\end{figure}
Figure \ref{fig:fig4} suggests that the efficiency of the proposed unit-length estimator surpasses that of the standard PCA estimator for heavy-tailed data. Interestingly, this happens even when the signal is not strong enough, i.e., when $\lambda$ is too small in the sense of Theorem \ref{theo:main_1}. Investigations reveals that in these cases we still have $\| w_n \| \approx 0 $ (as claimed in Theorem \ref{theo:main_1}), but that the convergence to the origin occurs from the ``right direction'' in such a manner that $w_n/\| w_n \| \rightarrow_p o_1$. This phenomenon is not predicted by our theoretical results and should be investigated further in future work. Note also that the same effect could not be observed in the previous simulation, where we plotted not $w_n/\| w_n \|$ but the raw values of the components of $w_n$ instead. We also observe that the difference between the methods under heavy-tailed data is more pronounced in higher-dimensional scenarios. Moreover, as expected, the efficiency of PCA increases as the number $\nu$ of degrees of freedom of the underlying $t$-distribution increases, i.e., the data has lighter tails and resembles the Gaussian distribution more. On the other hand, the behavior of our proposed method seems largely unaffected by the choice of~$\nu$. 




\subsection*{Acknowledgments}

The work of JV was supported by the Research Council of Finland under Grants 347501, 353769. This research was funded in whole or in part by the Austrian Science Fund (FWF) [10.55776/I5799]. 

\appendix

\section{Proofs of technical results}\label{sec:proofs}

\subsection{Proofs related to Section \ref{sec:introduction}}\label{sec:proofs_1}

The following lemmas prove the claims used to motivate our concept in the first paragraph of Section \ref{sec:introduction}.

\begin{lemma}\label{lem:nuclear_norm}
    For $a, b \in \mathbb{R}^p$, we have $\| a a' - b b' \|_* = \| a - b \| \| a + b \|$.    
\end{lemma}

\begin{proof}[Proof of Lemma \ref{lem:nuclear_norm}]
  The nuclear norm of a symmetric matrix is the sum of the absolute values of its eigenvalues. Hence, we next obtain expressions for the eigenvalues of the at most rank-two matrix $A := a a' - b b'$.

  Assume first that $a \nparallel b$, in which case $A$ has exactly two non-zero eigenvalues which we call $\phi_{+}$ and $\phi_{-}$. These satisfy $\phi_{+} + \phi_{-} = \mathrm{tr}(A)$ and $\phi_{+}^2 + \phi_{-}^2 = \mathrm{tr}(A^2)$, and by computing the traces and using the quadratic formula, we obtain
  \begin{align*}
      \phi_\pm = \frac{1}{2} \left\{ \| a \|^2 - \| b \|^2 \pm [ (\| a \|^2 - \| b \|^2)^2 - 4 (a'b)^2 + 4 \| a \|^2 \| b \|^2 ]^{1/2} \right\}.
  \end{align*}
  Since $\| a \|^2 \| b \|^2 \geq (a'b)^2$, we have that $\phi_{-} \leq 0$. Similarly, we obtain that $\phi_{+} \geq 0$, giving
  \begin{align*}
      \| a a' - b b' \|_* = \phi_{+} - \phi_{-} = [ (\| a \|^2 - \| b \|^2)^2 - 4 (a'b)^2 + 4 \| a \|^2 \| b \|^2 ]^{1/2},
  \end{align*}
  which is straightforwardly seen to be equal to $\| a - b \| \| a + b \| = \{ ( \| a \|^2 - 2a'b + \| b \|^2) ( \| a \|^2 + 2a'b + \| b \|^2) \}^{1/2}$. For the case $b = \lambda a$ for some $\lambda$, we have $\| a a' - b b' \|_* = | 1 - \lambda^2 | \| a a' \|_* = | 1 - \lambda^2 | \| a \|^2$ and $\| a - b \| \| a + b \| = | 1 - \lambda| | 1 + \lambda | \| a \|^2 = | 1 - \lambda^2 | \| a \|^2$, proving the claim. 
\end{proof}

\begin{lemma}\label{lem:interpretation}
    Fix $v \in \mathbb{R}^p$ and $R > 0$, and let $\mathcal{S}_R := \{ x \in \mathbb{R}^p \mid \| x \| = R \}$. Then, the unique minimizers of the function
    \begin{align*}
        x \mapsto \| x - v \| \| x + v \|
    \end{align*}
    in $\mathcal{S}_R$ are $\pm R v/\| v \|$.
\end{lemma}

\begin{proof}[Proof of Lemma \ref{lem:interpretation}]
    Without loss of generality, we instead minimize the function $h(x) := \| x - v \|^2 \| x + v \|^2$ which can be written as
    \begin{align*}
        h(x) = (\| x \|^2 + \| v \|^2)^2 - 4 (x'v)^2.
    \end{align*}
    Hence, for $x \in \mathcal{S}_R$, minimizing $h(x)$ is equivalent to maximizing $x \mapsto (x'v)^2$ which, by the Cauchy-Schwarz inequality, is achieved precisely at $\pm R v/\| v \|$.
\end{proof}

\subsection{Proofs related to Section \ref{sec:population}}\label{sec:proofs_2}

\begin{proof}[Proof of Theorem \ref{theo:moments}]
    The result follows directly from the alternative form of the objective function, $f_P(v) = \mathrm{E}( \| X X' - v v' \|_* - \| X X' \|_* )$, by using the reverse triangle inequality for the nuclear norm, $|  \| X X' - v v' \|_* - \| X X' \|_* | \leq \| v v' \|_* = \| v \|^2$.
\end{proof}

\begin{proof}[Proof of Theorem \ref{theo:existence}]
    The strategy of the proof is as follows: Lebesgue's dominated convergence theorem shows, using the same bound as in the proof of Theorem \ref{theo:moments}, that $f_P$ is continuous. We then search for a large enough $H > 0$ such that $f_P(v) > f_P(0)$ for all $v \notin \mathcal{K}_H := \{ v \in \mathbb{R}^p \mid \| v \| \leq H \}$. The continuity of $f_P$ and the compactness of the set $\mathcal{K}_H$ then together guarantee that $f_P$ admits a global minimizer in $\mathcal{K}_H$.
    
    For $R > 0$, we denote $\mathcal{B}(R) := \{ x \in \mathbb{R}^p \mid \| x \| \leq R \}$. Let next $G$ denote the distribution of $\| X \|$ and let $Q_q(G)$ be the $q$th  quantile of $G$. Fixing then $R_0 := Q_{0.75}(G)$, we see that $2 P \{ X \in \mathcal{B}(R_0) \} - 1 > 0$ holds. Then, by the triangle inequality,
    \begin{align*}
        & \mathrm{E} \{ \mathbb{I}[X \in \mathcal{B}(R_0)] ( \| X - v \| \| X + v \| - \| X \|^2 ) \} \\
        \geq& \mathrm{E} \{ \mathbb{I}[X \in \mathcal{B}(R_0)] ( \| v \|^2 - 2 \| v \| \| X \| - \| X \|^2 ) \} \\
        \geq& P \{ X \in \mathcal{B}(R_0) \}  \| v \|^2 - 2 \| v \| R_0 - R_0^2.
    \end{align*}
    The first inequality above is obtained by first noting, using the notation of the proof of Theorem \ref{theo:moments}, that
    \begin{align*}
    \| X - v \| \| X + v \| = \left\| \begin{pmatrix}
            \| X \|^2 - \| v \|^2 \\
            A_0^{1/2}
        \end{pmatrix} \right\| \geq | \| X \|^2 - \| v \|^2 |.
    \end{align*}
    and then observing that $| \| X \|^2 - \| v \|^2 | \geq \| v \|^2 - 2 \| v \| \| X \|$, which can be verified by separately checking the cases $\| X \| > \| v \|$ and $\| X \| \leq \| v \|$.
    
    Similarly, using the nuclear norm form of the objective and the bound in the proof of Theorem \ref{theo:moments}, 
    \begin{align*}
        & \mathrm{E} \{ \mathbb{I}[X \notin \mathcal{B}(R_0)] ( \| X - v \| \| X + v \| - \| X \|^2 ) \} \\
        \geq& 
        - \mathrm{E} \{ \mathbb{I}[X \notin \mathcal{B}(R_0)]  \| v \|^2  \} \\
        =& -P \{ \mathbb{R}^p \setminus \mathcal{B}(R_0) \} \| v \|^2.
    \end{align*}
    Consequently, for a fixed $v \in \mathbb{R}^p$, our objective function satisfies,
    \begin{align*}
        f_P(v) \geq& (2 P \{ X \in \mathcal{B}(R_0) \} - 1) \| v \|^2 - 2 \| v \| R_0  - R_0^2.
    \end{align*}
    As a function of $\| v \|$, the previous quantity is dominated by the second order term $\| v \|^2$. Let next $H := H_0 + 1$ where $H_0$ is the largest root of the polynomial 
    \begin{align*}
        (2 P \{ X \in \mathcal{B}(R_0) \} - 1) \| v \|^2 - 2 \| v \| R_0 +  - R_0^2 - f_P(0) = 0.
    \end{align*}
    Consequently, we have 
    \begin{align*}
        (2 P \{ X \in \mathcal{B}(R_0) \} - 1) \| v \|^2 - 2 \| v \|  R_0 - R_0^2 > f_P(0).
    \end{align*}
    for all $v \not\in \mathcal{K}_H$. The claim now follows as described in the beginning of the proof.
\end{proof}

We next establish a collection of auxiliary results, using which Theorem~\ref{theo:main_1} is then proven.

\begin{lemma}\label{lem:aux_1}
    Let $u_1, u_2 \in \mathbb{S}^{p - 1}$ be such that $(o_1'u_1)^2 > (o_1'u_2)^2$. Then,
    \begin{align*}
        f_X(u_1) > f_X(u_2),
    \end{align*}
    for any $X \sim P \in \mathcal{E}(O, \Lambda)$.
\end{lemma}

\begin{proof}[Proof of Lemma \ref{lem:aux_1}]
     Without loss of generality, we take $\sigma = 1$, meaning that the final $p - 1$ diagonal elements of $\Lambda$ equal 1. By Theorem 2.9 and formula (2.43) in \cite{fang1990symmetric}, the density function of $X$ is
\begin{align*}
    f_X(x) = \frac{\Gamma(p/2)}{2 \pi^{p/2} |\Lambda|} \frac{f_R((x'O \Lambda^{-2} O' x)^{1/2})}{(x'O \Lambda^{-2} O' x)^{(p - 1)/2}}.
\end{align*}
    By moving to the coordinate system corresponding to the columns of $O$, we may without loss of generality assume that $O = I_p$, in which case the claim is equivalent to showing that
    \begin{align*}
        \frac{f_R((a'\Lambda^{-2} a)^{1/2})}{(a' \Lambda^{-2} a)^{(p - 1)/2}} > \frac{f_R((b'\Lambda^{-2} b)^{1/2})}{(b' \Lambda^{-2} b)^{(p - 1)/2}},
    \end{align*}
    for all $a, b \in \mathbb{S}^{p - 1}$ such that $a_1^2 > b_1^2$. Now $a'\Lambda^{-2} a = (1/\lambda^{2} - 1)a_1^2 + 1$, which is a strictly decreasing function of $a_1^2$ since $\lambda > 1$. The claim then follows from our assumption that $r \mapsto f_R(r)/r^{p - 1}$ is strictly decreasing.
    \end{proof}

    In the following, we consider probability distributions on $p-1$-dimensional spheres embedded in the $p$-dimensional real space. The reference measure on such spheres is the conventional surface measure denoted by $dS$. Moreover, let $f_{\Vert X \Vert^2, (u'X)^2}(x,y)$ be the joint density function, $f_{\Vert X\Vert^2}(x)$ the marginal density function and $f_{(u'X)^2 | \Vert X \Vert^2}(y|x)$ the conditional density function related to the random variables $\Vert X \Vert^2$ and $(u'X)^2$. The next lemma shows that the projection $(o_1'X)^2$ conditionally stochastically dominates projections $(u'X)^2$ onto any other unit vector given $\Vert X \Vert^2$.

\begin{lemma}\label{lem:aux_2}
    Let $X \sim P \in \mathcal{E}(O, \Lambda)$ and let $u \in \mathbb{S}^{p - 1}$ be such that $u \neq \pm o_1$. Then,
    \begin{align*}
        P( (o_1' X)^2 > r \mid \| X \|^2 = R ) > P( (u' X)^2 > r \mid \| X \|^2 = R ),
    \end{align*}
    for all $R > 0$ and $r \in (0, R)$.
\end{lemma}

\begin{proof}[Proof of Lemma \ref{lem:aux_2}]
    Without loss of generality, the scale equivariance of the problem allows us to consider the case $R = 1$ only. For $u \in \mathbb{S}^{p - 1}$ and $r \in (0, 1)$, let $C(u, r) := \{ v \in \mathbb{S}^{p - 1} \mid (u'v)^2 \geq r \}$ denote the union of spherical caps around $u$ and $-u$ with a fixed radius defined by $r$. Note that since $r \in (0, 1)$, the measure of $C(u, r)$ is strictly positive, but the set does not cover the whole unit sphere. Then
    \begin{equation*}
    \begin{split}
        P( (u' X)^2 > r \mid \| X \|^2 = 1 ) &= \int_r^1 f_{(u'X)^2 | \Vert X \Vert^2}(y|1)dy\\
        &= \frac{1}{f_{\Vert X\Vert^2}(1)}\int_r^1 f_{\Vert X \Vert^2, (u'X)^2}(1,y) dy\\
        &= \frac{1}{f_{\Vert X\Vert^2}(1)} \int_{C(u,r)} f_X dS
    \end{split}
    \end{equation*}
    Based on the above, our strategy is to show that there exists a measure preserving bijection $H: C(u,r) \mapsto C(o_1,r)$ such that $f_X(H(v)) \geq f_X(v)$ for all $v$ with a strict inequality in a set of positive measure. For $v\in C(u,r)\cap C(o_1,r)$, we set $H(v) = v$.
    Furthermore, we denote the matrix corresponding to the reflection w.r.t to the hyperplane orthogonal to the vector $o_1 - u$ by
    \begin{align*}
        H_1 := I_p - 2\frac{(o_1 - u)(o_1 - u)'}{\|o_1 - u\|^2}.
    \end{align*}
    Thus, in particular, $H_1 o_1 = u$, $H_1 u = o_1$ and $\| H_1 v \| = 1$ for all $v \in \mathbb{S}^{p - 1}$. In addition, as a reflection, $H_1$ is a measure preserving bijective mapping. Now, for $v\in C(u,r)\setminus C(o_1,r)$, we set $H(v) = H_1v$. Since $(u'v)^2 \geq r$ and $(o_1' v)^2 < r$, we obtain that 
    \begin{equation}
    \label{eq:projection}
        (o_1' H(v))^2 = (o_1' H_1v)^2=(u'v)^2 \geq r
    \end{equation} 
    and $(u'H(v))^2 = (u'H_1v)^2= (o_1'v) < r$. Consequently, $H(v) \in C(o_1,r) \setminus C(u,r)$ for $v\in C(u,r)\setminus C(o_1,r)$, and furthermore $H: C(u,r) \mapsto C(o_1,r)$ is a measure preserving bijection.

    Since $r \in (0, 1)$, the measure of $ C(u, r) \setminus C(o_1, r)$ is strictly positive. For every $v\in C(u, r) \setminus C(o_1, r)$, it holds that $(o_1'v)^2 < r$. Moreover, by \eqref{eq:projection}, $ (o_1' H(v))^2 \geq r$. We conclude by Lemma \ref{lem:aux_1} that 
    \begin{align*}
        f_X(H(v)) > f_X(v).
    \end{align*}

\end{proof}

\begin{lemma}\label{lem:aux_3}
   Let $X \sim P \in \mathcal{E}(O, \Lambda)$. Denoting $\mathcal{A} := \{ (x, y) \in \mathbb{R}_{\geq 0}^2 \mid  y \leq x \}$, let $g: \mathcal{A} \to \mathbb{R}$ be a continuous function satisfying the following conditions.
   \begin{enumerate}
       \item $\mathrm{E} \{ g( \| X \|^2, (o_1'X)^2)\}< \infty$
       \item $g$ is strictly decreasing in its second argument.
       \item For every fixed $x >0$, $g$ is differentiable w.r.t. its second argument on $\mathcal{A}_x := \{y\in(0,x)\}$.
   \end{enumerate}
  Then, for all $u \in \mathbb{S}^{p - 1}$such that $u \neq \pm o_1$, we have
    \begin{align*}
        \mathrm{E} \{ g(  \| X \|^2, (o_1'X)^2)\} < \mathrm{E} \{ g( \| X \|^2, (u'X)^2)\}.
    \end{align*}
\end{lemma}

\begin{proof}[Proof of Lemma \ref{lem:aux_3}]
By Item 1, the conditional expectation $\mathrm{E} \{ g(  \| X \|^2, (o_1'X)^2) \mid \Vert X \Vert^2\}$ exists as an almost surely finite random variable. By integration by parts and Item 3,
\begin{equation*}
\begin{split}
\label{eq:condexpectations}
&\mathrm{E} \{ g(  \| X \|^2, (o_1'X)^2) \mid \Vert X \Vert^2 = R\} = \int_0^R g(R, y) f_{(o_1'X)^2 | \Vert X \Vert^2}(y|R) dy \\
&= \bigg\vert_0^R g(R, y) F_{(o_1'X)^2 | \Vert X \Vert^2}(y|R) - \int_0^R \frac{\partial}{\partial y}g(R, y) F_{(o_1'X)^2 | \Vert X \Vert^2}(y|R) dy\\
&= g(R, R) - \int_0^R \frac{\partial}{\partial y}g(R, y) dy + \int_0^R \frac{\partial}{\partial y}g(R, y) P((o_1'X)^2 > y \mid \Vert X \Vert^2 = R)dy \\
&= g(R,0) + \int_0^R \frac{\partial}{\partial y}g(R, y) P((o_1'X)^2 > y \mid \Vert X \Vert^2 = R)dy\\
&< g(R,0) + \int_0^R \frac{\partial}{\partial y}g(R, y) P((u'X)^2 > y \mid \Vert X \Vert^2 = R)dy\\
&= \mathrm{E} \{g( \| X \|^2,  (u'X)^2) \mid \Vert X \Vert^2 = R\},
\end{split}
\end{equation*}
where the inequality follows by Lemma \ref{lem:aux_2} and Item 2. 
Now, the result follows by taking expectations on both sides together with the tower property of conditional expectations.

\end{proof}

\begin{theorem}\label{theo:direction}
The minimizers of $f_P$ are vectors of the form $v= \pm \psi o_1$, where $\psi \geq 0.$
\end{theorem}

\begin{proof}[Proof of Theorem \ref{theo:direction}]
Let $u\in\mathbb{S}^{p-1}$ and $t \geq 0$. Minimizing $f_P$ is equivalent to minimizing 
\begin{equation*}
\begin{split}
&\mathrm{E}\{\Vert X- \sqrt{t}u \Vert \Vert X+ \sqrt{t}u\Vert- \Vert X \Vert^2\} = \mathrm{E}\left\{\sqrt{ \Vert X- \sqrt{t}u \Vert^2 \Vert X+ \sqrt{t}u\Vert^2 } - \Vert X \Vert^2 \right\}\\
&= \mathrm{E}\left\{\sqrt{ (\Vert X \Vert^2 + t -2\sqrt{t} u'X)(\Vert X \Vert^2 + t +2\sqrt{t} u'X) } - \Vert X \Vert^2 \right\} \\
&= \mathrm{E}\left\{\sqrt{(\Vert X\Vert^2 + t)^2 - 4t(u'X)^2} - \Vert X \Vert^2 \right\}
\end{split}
\end{equation*}
with respect to $u$ and $t$. Define $g_t: \mathcal{A} \to \mathbb{R}$ by
$$g_t(x,y) = \sqrt{(x + t)^2 - 4ty} - x. $$
As a direct result of Theorem \ref{theo:moments} we have that $\mathrm{E} \{g_t( \| X \|^2, (o_1'X)^2)\}< \infty.$ In addition, $g_t$ is strictly decreasing with respect to its second argument whenever $t > 0$, and the nominator of
\begin{equation*}
\frac{\partial}{\partial y}g_t(x,y) = \frac{-2t}{\sqrt{(x+t)^2 - 4ty}},
\end{equation*}
is zero only if $x=y=t$. Now, Lemma \ref{lem:aux_3} yields
\begin{equation*}
\begin{split}
&\mathrm{E}\left\{\sqrt{(\Vert X\Vert^2 + t)^2 - 4t(o_1'X)^2} - \Vert X \Vert^2 \right\}= \mathrm{E} \{ g_t( \| X \|^2, (o_1'X)^2)\}\\ 
&<   \mathrm{E} \{ g_t(  \| X \|^2, (u'X)^2)\}= \mathrm{E}\left\{\sqrt{(\Vert X\Vert^2 + t)^2 - 4t(u'X)^2} - \Vert X \Vert^2 \right\}
\end{split}
\end{equation*}
for every fixed $t > 0$. This means that on any sphere of radius $r$, the objective function is minimized by $\pm ro_1$. Hence, the minimizers of $f_P$ are of the form $v= \pm\psi o_1$, where $\psi\in[0,\infty).$ 
\end{proof}

The following lemma, together with Theorem \ref{theo:direction}, provides the final tools for the proof of Theorem \ref{theo:main_1}.

\begin{lemma}\label{lem:convexity_of_h}
    Let $a, b \in \mathbb{R}^p$ be arbitrary and let $h: \mathbb{R}_{\geq 0} \to \mathbb{R}$ be defined as
    \begin{align*}
        h(t) = \| a - t^{1/2} b \| \| a + t^{1/2} b \|.
    \end{align*}
    Then $h$ is convex and, furthermore, it is strictly convex unless $a = \theta b$ for some $\theta \in \mathbb{R}$.
\end{lemma}

\begin{proof}[Proof of Lemma \ref{lem:convexity_of_h}]
    By the continuity of $h$, it is sufficient to check that midpoint convexity, $2 h( (x + y)/2 ) \leq h(x) + h(y) $ holds for all $x, y \geq 0$, $x \neq y$. Squaring both sides, this inequality is equivalent to
    \begin{align*}
        & 4 \left[ \left\{ \| a \|^2 + \left( \frac{x + y}{2} \right) \| b \|^2 \right\}^2 - 4 \left( \frac{x + y}{2} \right) (a'b)^2 \right] \\
        & \leq ( \| a \|^2 + x \| b \|^2)^2 - 4 x (a'b)^2 + ( \| a \|^2 + y \| b \|^2)^2 - 4 y (a'b)^2 \\
        &+ 2 \| a - x^{1/2} b \| \| a + x^{1/2} b \| \| a - y^{1/2} b \| \| a + y^{1/2} b \|.
    \end{align*}
    This simplifies to
    \begin{align*}
        & (\| a \|^2 + x \| b \|^2) (\| a \|^2 + y \| b \|^2) \\
        & \leq \| a - x^{1/2} b \| \| a + x^{1/2} b \| \| a - y^{1/2} b \| \| a + y^{1/2} b \| + 2 x (a'b)^2 + 2 y (a'b)^2.
    \end{align*}
    Using next the notation $ A := (\| a \|^2 + x \| b \|^2)$, $ B := (\| a \|^2 + y \| b \|^2)$, $C := x (a'b)^2$, $D := y (a'b)^2$, the above can be written as
    \begin{align}\label{eq:ABCD_1}
        A B - 2C - 2D \leq (A^2 - 4C)^{1/2} (B^2 - 4D)^{1/2}.
    \end{align}
    The RHS of \eqref{eq:ABCD_1} is non-negative, meaning that if $ A B - 2C - 2D < 0 $, we are done. Whereas, if $ A B - 2C - 2D \geq 0 $, we may square both sides of \eqref{eq:ABCD_1} and our claim is then equivalent to
    \begin{align*}
        (C - D)^2 \leq (A - B) (BC - AD),
    \end{align*}
    which can be written as
    \begin{align*}
        (a'b)^4 (x - y)^2 \leq \| a \|^2 \| b \|^2 (a'b)^2 (x - y)^2.
    \end{align*}
    By the Cauchy-Schwarz inequality, the above holds for all $a, b$. Furthermore, since $x \neq y$, the equality is reached if and only if $a = \theta b$ for some $\theta \in \mathbb{R}$.
\end{proof}

\begin{proof}[Proof of Theorem \ref{theo:main_1}]
    From Theorem \ref{theo:direction} we obtain that the minimizers of $f_P$ are vectors of the form $v= \pm \psi o_1$, where $\psi \geq 0$. Hence, it remains to study the values that the scalar $\psi$ can take. We let $h_P: \mathbb{R}_{\geq 0} \to \mathbb{R}$ be defined as
    \begin{align*}
        h_{P}(t) := \mathrm{E} \left( \| X - \sqrt{t} o_1 \| \| X + \sqrt{t} o_1 \| - \| X \|^2 \right)
    \end{align*}
    and denote then
$$f(t,X) = \sqrt{(\Vert X\Vert^2 + t)^2 - 4t(o_1'X)^2} - \Vert X \Vert^2 $$
so that $h_P(t) = \mathrm{E}f(t,X).$ The partial derivative
$$f_t(t,X) = \frac{(\| X \|^2 + t ) - 2(X' o_1)^2}{\{ (\| X \|^2 + t )^2 - 4 t (X' o_1)^2 \}^{1/2}}$$
is defined unless $\Vert X \Vert^2 = X_1^2 = t$. That is, $f_t(t,X)$ is defined on the set $\{\Vert X \Vert^2 \neq X_1^2 \}$ of full measure for all $t$. Furthermore, in this set, the absolute value of the partial derivative is bounded by one, since
$$\{(\| X \|^2 + t ) - 2(X' o_1)^2\}^2 \leq  (\| X \|^2 + t )^2 - 4 t (X' o_1)^2 $$
if and only if
$$4(X' o_1)^4 - 4(\| X \|^2 + t )(X' o_1)^2 \leq - 4 t (X' o_1)^2.$$
This is equivalent to 
$$(X' o_1)^2((X' o_1)^2- \| X \|^2) \leq 0,$$
which always holds. Now, Leibniz integral rule gives $h'_P(t) = \mathrm{E}f_t(t,X)$.
    
    Moreover, the continuity of $h_P$ and the absolute continuity of the distribution of $X$ imply in conjunction with Lemma \ref{lem:convexity_of_h} that $h_{P}$ is strictly convex. Now, if $\phi_1(S_P) \leq 1/2$, then the right derivative satisfies
    \begin{align*}
        h_{P}'(0) = 1 - 2 o_1' S_P o_1 \geq 0.
    \end{align*}
    Hence, on all rays starting from the origin, the map $f_P$ is minimized in the origin, making $v = 0$ its unique minimizer. On the other hand, if $\phi_1(S_P) > 1/2$, then $h_{P}'(0) < 0$ and Lebesgue's DCT further shows that $h'_{P}(t) \rightarrow 1$ as $t \rightarrow \infty$. Hence, in this case $h_P$ has a unique positive minimizer, say $t \equiv \psi > 0$, and the minimizers of $f_P$ are precisely the vectors $v = \pm \psi o_1$.
\end{proof}


Next we provide proofs for the asymptotic ($p \rightarrow \infty$ or $\lambda \rightarrow \infty$) results of Section \ref{sec:population}.

\begin{proof}[Proof of Lemma \ref{lem:integral}]
    Without loss of generality, we assume that $\sigma = 1$. Since the spatial sign covariance matrix is orthogonally equivariant, the eigenvalues are invariant under orthogonal transformations, and we may assume that $O = I$. Now, the density function of $X$ takes the form
    \begin{equation}
    \label{eq:densityofX}
    f_X(x) = \prod_{j=1}^p \frac{1}{\lambda_j}  g(\Vert \Lambda^{-1} x \Vert^2) = \frac{1}{\lambda}  g\left( \frac{x_1^2}{\lambda^2} + x_2^2 + \dots + x_p^2 \right),
    \end{equation}
where $f_Z(z) = g(\Vert z \Vert^2)$.
This gives $V_0 = \frac{\mathrm{diag}(\lambda^2, 1, \dots, 1)}{\lambda^2 + (p-1)}$ in \cite{durre2016eigenvalues}. Then, Proposition 3 (note that $\lambda$ has here different interpretation than in our manuscript) gives
$$\tau = \frac{\lambda^2}{2(\lambda^2 + (p-1))} \int_0^\infty \frac{1}{(1+\frac{\lambda^2}{\lambda^2+(p-1)}x)^\frac32 {(1+\frac{1}{\lambda^2+(p-1)}x)^\frac{p-1}{2}}} dx.$$
Denote $a= 1/(\lambda^2 + (p-1))$ and $y = ax$. Then
$$\tau = \frac{\lambda^2}{2}a \int_0^\infty \frac{1}{(1+\lambda^2y)^\frac32(1+y)^\frac{p-1}{2}} \frac{1}{a} dy.$$
\end{proof}

\begin{proof}[Proof of Theorem \ref{lemma:tau}]
\begin{enumerate}[align=left, leftmargin=5pt, labelwidth=-0.4em, label=(\arabic*)]
    \item The integrand of \eqref{eq:durre_integral} is decreasing for every $x$ showing the first claim. 
    For the second claim, we set $x=u/p$ yielding
    $$p\tau = \frac{\lambda^2}{2}\int_0^\infty \frac{p(1+x)^{\frac{1}{2}}}{(1+\lambda^2x)^{\frac{3}{2}}(1+x)^\frac{p}{2}} dx = \frac{\lambda^2}{2}\int_0^\infty \frac{(1+\frac{u}{p})^{\frac{1}{2}}}{(1+\lambda^2\frac{u}{p})^{\frac{3}{2}}(1+\frac{u}{p})^\frac{p}{2}} du,  $$
    where 
    \begin{equation}
    \label{eq:boundthis}
    \frac{(1+\frac{u}{p})^{\frac{1}{2}}}{(1+\lambda^2\frac{u}{p})^{\frac{3}{2}}(1+\frac{u}{p})^\frac{p}{2}} \leq \frac{1}{(1+\frac{u}{p})^\frac{p}{2}}.
    \end{equation}
    Let $h(p) = (1+u/p)^{-\frac{p}{2}}.$ Then
    \begin{equation}
    \label{eq:logD}
    \begin{split}
    \frac{d}{d p} \log h(p) &= \frac{d}{d p} \left[-\frac{p}{2} \log \left(1+\frac{u}{p}\right)\right] = -\frac12\log\left(1+\frac{u}{p}\right) + \frac{\frac{p}{2}\frac{u}{p^2}}{1+\frac{u}{p}}\\
    &= -\frac12\log\left(1+\frac{u}{p}\right) + \frac{u}{2(p+u)} =: \xi(p).
    \end{split}
    \end{equation}
    Furthermore,
    \begin{equation*}
        \begin{split}
            \xi'(p) &= \frac{\frac{u}{2p^2}}{1+\frac{u}{p}} - \frac{u}{2(p+u)^2} = \frac{1}{p} \frac{u}{2(p+u)}- \frac{u}{2(p+u)^2}\\
            &= \frac{u}{2(p+u)}\left(\frac{1}{p}- \frac{1}{p+u}\right) > 0
        \end{split}
    \end{equation*}
    for all $u>0$. Consequently, $\xi(p)$ is strictly increasing, and $\lim_{p\to\infty} \xi(p) = 0$. Hence, \eqref{eq:logD} is strictly negative for all $u>0$ and $p$ implying that $h(p)$ is strictly decreasing for all $u>0$. This gives $h(p) \leq h(4)$ for all $p \geq 4$, and by \eqref{eq:boundthis},
    $$ \frac{(1+\frac{u}{p})^{\frac{1}{2}}}{(1+\lambda^2\frac{u}{p})^{\frac{3}{2}}(1+\frac{u}{p})^\frac{p}{2}} \leq \frac{1}{(1+\frac{u}{4})^2} $$
    for all $u>0$ and $p \geq 4$. The dominated convergence theorem concludes that
    $$\lim_{p\to\infty} p\tau = \frac{\lambda^2}{2} \int_0^\infty  \frac{1}{e^{\frac{u}{2}}} dx = \frac{\lambda^2}{2}\cdot 2 = \lambda^2.$$  
    \item First, we notice that $$\lim_{p\to\infty} p \tau_0 = \lim_{p\to\infty} \frac{p(1-\tau)}{p-1}= \lim_{p\to\infty} \frac{p}{p-1} - \lim_{p\to\infty} \frac{p\tau}{p-1} = 1.$$
    For monotonicity, we want to show that
    $$\frac{1-\tau(\lambda,p)}{p-1} > \frac{1-\tau(\lambda,p+1)}{p}, $$
    which is equivalent to 
    \begin{align*}
        & p\tau(\lambda,p) - (p-1)\tau(\lambda,p+1) \\
        =& \frac{\lambda^2}{2} \int_0^\infty \left(p- \frac{p-1}{(1+x)^{\frac{1}{2}}} \right)\frac{(1+x)^{\frac{1}{2}}}{(1+\lambda^2x)^{\frac{3}{2}}(1+x)^\frac{p}{2}} dx \\
        <& 1.
    \end{align*}
    The change of variable $x = u/\lambda^2$ turns this to
    \begin{equation}
    \label{eq:lessthan1}    
    \frac{1}{2}\int_0^\infty \left(p- \frac{p-1}{(1+\frac{u}{\lambda^2})^{\frac{1}{2}}} \right)\frac{1}{(1+u)^{\frac{3}{2}}(1+\frac{u}{\lambda^2})^\frac{p-1}{2}} du < 1,
    \end{equation}
    where the integrand is increasing in $\lambda$ if and only if
    $$h_p(\lambda) = \frac{p}{(1+\frac{u}{\lambda^2})^\frac{p-1}{2}}- \frac{p-1}{(1+\frac{u}{\lambda^2})^\frac{p}{2}}$$
    is increasing in $\lambda$. We obtain that
    $$h'_p(\lambda) = \frac{p(p-1)\frac{u}{\lambda^3}}{(1+\frac{u}{\lambda^2})^{\frac{p+1}{2}}}- \frac{p(p-1)\frac{u}{\lambda^3}}{(1+\frac{u}{\lambda^2})^{\frac{p}{2}+1}} > 0$$
    iff $(1+\frac{u}{\lambda^2})^\frac12 -1 > 0,$ which holds true for all $u > 0$ and $\lambda >1$. We conclude that the integral in \eqref{eq:lessthan1} is strictly bounded by
    $$\frac{1}{2}\int_0^\infty \frac{1}{(1+u)^\frac32 }du = 1.$$

    \item Via the change of variable $x=u/\lambda^2$, we obtain that
    $$\tau = \frac{1}{2}\int_0^\infty \frac{1}{(1+u)^\frac{3}{2}(1+\frac{u}{\lambda^2})^{\frac{p-1}{2}}} du,$$
    where the integrand is strictly increasing in $\lambda$ for almost every $u$. Furthermore, the integrand is bounded by $\frac{1}{(1+u)^\frac{3}{2}}$. The dominated convergence theorem then gives
    $$\lim_{\lambda\to\infty} \tau = \frac{1}{2}\int_0^\infty\frac{1}{(1+u)^\frac{3}{2}} du = \frac{1}{2} \Big|_0^\infty \frac{-2}{\sqrt{1+u}} = 1.$$

    \item By Items 1 and 3, $\lambda_p^*$ is strictly increasing in $p$. For the latter claim, we investigate limits
    \begin{equation}
    \begin{split}
        \label{eq:rateof}
        \tau(p^a, p) &= \frac{p^{2a}}{2} \int_0^\infty \frac{(1+x)^\frac12}{(1+p^{2a}x)^\frac32(1+x)^\frac{p}{2}} dx\\
        &= \frac12 \int_0^\infty \frac{(1+\frac{u}{p^{2a}})^\frac12}{(1+u)^\frac32(1+\frac{u}{p^{2a}})^\frac{p}{2}} du,
        \end{split}
    \end{equation}
    where $a >0$ and $x=u/p^{2a}.$ The integrand admits the uniform upper bound
    $$\frac{1}{(1+u)^\frac32(1+\frac{u}{p^{2a}})^\frac{p-1}{2}} \leq \frac{1}{(1+u)^\frac32}.$$
    Let $q = p^{2a}$. Then 
    \begin{equation}
    \label{eq:limit1}
    \lim_{p\to\infty}\left(1+\frac{u}{p^{2a}}\right)^\frac{p}{2}= \lim_{q\to\infty}\left[\left(1+\frac{u}{q}\right)^{q^{\frac{1}{2a}}}\right]^\frac12,
    \end{equation}
    where
    \begin{equation}
    \label{eq:limit2}
    \lim_{q\to\infty}\log \left(1+\frac{u}{q}\right)^{q^{\frac{1}{2a}}} = \lim_{q\to\infty} q^{\frac{1}{2a} -1}\log \left(1+\frac{u}{q}\right)^q
    \end{equation}
    with 
    $$\lim_{q\to\infty}\log \left(1+\frac{u}{q}\right)^q = u. $$
    Hence, \eqref{eq:limit2} converges to zero for $a> 1/2$ and diverges to infinity for $a < 1/2$. Consequently, \eqref{eq:limit1} converges to one for $a>1/2$ and diverges to infinity for $a<1/2.$ Furthermore, for \eqref{eq:rateof} this gives
    $$\lim_{p\to\infty} \tau(p^a,p) = \frac12\int_0^\infty \frac{1}{(1+u)^\frac32} du =1, \quad\text{when } a>\frac12,$$
    and $\lim_{p\to\infty} \tau(p^a,p) = 0$, when $a<1/2.$ This means that asymptotically
    $$p^{\frac{1}{2} - \varepsilon} < \lambda_p^* < p^{\frac{1}{2}+\varepsilon}. $$
    for all $\varepsilon >0$. 

    \item Now, let $a = 1/2$. Similarly as above, we obtain that
    \begin{equation*}
       \tau(Cp^\frac12, p) =  \frac12 \int_0^\infty \frac{(1+\frac{u}{C^2p})^\frac12}{(1+u)^\frac32(1+\frac{u}{C^2p})^\frac{p}{2}} du 
       \overset{p\to\infty}{\longrightarrow} \frac12 \int_0^\infty \frac{1}{(1+u)^\frac32 e^{\frac{u}{2C^2}}} du, 
    \end{equation*}
    where the limit is strictly increasing in $C$. Moreover, setting $y = (1+u)^\frac12$ gives $du = 2ydy$ and $u =y^2-1$, and
    $$\int_0^\infty \frac{1}{(1+u)^\frac32 e^{\frac{u}{2C^2}}} du = 2e^{\frac{1}{2C^2}}\int_1^\infty \frac{1}{y^2e^{\frac{y^2}{2C^2}}} dy.  $$
    Furthermore, the change of variable $v = \frac{y}{\sqrt{2}C}$ gives $y^2 = 2C^2v^2$ and $dy = \sqrt{2}C dv$. Then, by integration by parts
    \begin{equation*}
        \begin{split}
    2 e^{\frac{1}{2C^2}}\int_1^\infty \frac{1}{y^2e^{\frac{y^2}{2C^2}}} dy &= \frac{\sqrt{2}}{C} e^{\frac{1}{2C^2}}\int_{\frac{1}{\sqrt{2}C}}^\infty v^{-2} e^{-v^2} dv \\
    &= \frac{\sqrt{2}}{C} e^{\frac{1}{2C^2}}\left( \Big |_{\frac{1}{\sqrt{2}C}}^\infty -v^{-1} e^{-v^2} - 2\int_{\frac{1}{\sqrt{2}C}}^\infty e^{-v^2} dv    \right)\\
    &= \frac{\sqrt{2}}{C} e^{\frac{1}{2C^2}}\left(\sqrt{2}C e^{-\frac{1}{2C^2}} - \sqrt{\pi} \mathrm{erfc}\left(\frac{1}{\sqrt{2} C}\right) \right),
    \end{split}
    \end{equation*}
    where $\mathrm{erfc}$ is the complementary error function. Thus
    $$ \lim_{p\to\infty} \tau(Cp^\frac12, p) = 1- \sqrt{\frac{\pi}{2}} \frac{e^{\frac{1}{2C^2}}}{C} \mathrm{erfc}\left(\frac{1}{\sqrt{2} C}\right) = \frac12 $$
    has a unique solution whose approximate value is $C \approx 1.633978435$.
\end{enumerate}
\end{proof}

\begin{proof}[Proof of Theorem \ref{thm:normofminimizer}]
Part (1) follows directly from Theorem \ref{lemma:tau}. For the second, we assume for convenience that $\sigma = 1$, the general case being proven in an identical manner. In addition, we select the coordinate system corresponding to $O=I$, and we recall from the proof of Theorem \ref{theo:main_1} that $h_P$ is strictly convex. Thus, when $\lambda > \lambda^*_p$, there exists $\psi(\lambda,p) >0$ such that $h'_P(t) < 0$ for all $t < \psi(\lambda,p)$ and $h'_P(t) > 0$ for all $t > \psi(\lambda,p)$. Let $X = \mathrm{diag}(\lambda, 1, \ldots, 1) Z$ be fixed and recalling the proof of Theorem \ref{theo:main_1}, let
$$g(t,\lambda, Z) := f_t(t,X) = \frac{-\lambda^2Z_1^2 + t +\sum_{j=2}^p Z_j^2}{\sqrt{(\lambda^2Z_1^2+ t +\sum_{j=2}^p Z_j^2 )^2 - 4t\lambda^2Z_1^2}},$$
which is defined on the set $\{\sum_{j=2}^p Z_j^2 \neq 0\}$ of full measure for all $t$ and $\lambda$, and is uniformly bounded by one in this set. Now, in the null set $\{Z_1=0\}$, we obtain that $g(t,\lambda, Z) =1.$ Let $Z_1 >0$ and $\epsilon >0$, then
\begin{equation*}
\begin{split}
g(\lambda^{2+\epsilon},\lambda, Z) &= \frac{-\lambda^2Z_1^2 + \lambda^{2+\epsilon} +\sum_{j=2}^p Z_j^2}{\sqrt{(\lambda^2Z_1^2+ \lambda^{2+\epsilon} +\sum_{j=2}^p Z_j^2 )^2 - 4\lambda^{2+\epsilon}\lambda^2Z_1^2}}\\
&= \frac{\lambda^{2+\epsilon}\left(-\frac{Z_1^2}{\lambda^\epsilon} + \frac{\sum_{j=2}^p Z_j^2}{\lambda^{2+\epsilon}} + 1 \right)}{\lambda^{2+\epsilon}\sqrt{\left(\frac{Z_1^2}{\lambda^\epsilon} + \frac{\sum_{j=2}^p Z_j^2}{\lambda^{2+\epsilon}} +1\right)^2 - \frac{4Z_1^2}{\lambda^\epsilon}}} \overset{\lambda\to\infty}{\longrightarrow} 1.
\end{split}
\end{equation*}
Similarly,
\begin{equation*}
\begin{split}
g(\lambda^{2-\epsilon},\lambda, Z) 
&= \frac{\lambda^{2}\left(-Z_1^2 + \frac{\sum_{j=2}^p Z_j^2}{\lambda^{2}} + \frac{1}{\lambda^\epsilon} \right)}{\lambda^{2}\sqrt{\left(Z_1^2 + \frac{\sum_{j=2}^p Z_j^2}{\lambda^{2}} + \frac{1}{\lambda^\epsilon} \right)^2 - \frac{4Z_1^2}{\lambda^\epsilon}}} \overset{\lambda\to\infty}{\longrightarrow} -1.
\end{split}
\end{equation*}
Consequently, by the dominated convergence theorem 
$$\lim_{\lambda\to\infty} h'_P(\lambda^{2+\epsilon}) = \mathrm{E} \left\{ \lim_{\lambda\to\infty} g(\lambda^{2+\epsilon},\lambda, Z) \right\} = \mathrm{E}(1) = 1,$$
and 
$$\lim_{\lambda\to\infty} h'_P(\lambda^{2-\epsilon}) = \mathrm{E} \left\{ \lim_{\lambda\to\infty}  g(\lambda^{2-\epsilon},\lambda, Z) \right\} = \mathrm{E}(-1) = -1.$$
Furthermore, when the rate is exactly $\lambda^2$, we obtain that
\begin{equation*}
\begin{split}
\lim_{\lambda\to\infty} h'_P(\lambda^{2}) &=  \mathrm{E} \left\{ \lim_{\lambda\to\infty} g(\lambda^{2},\lambda, Z) \right\} \\
&= \mathrm{E} \left\{ \lim_{\lambda\to\infty}\frac{\lambda^{2}\left(-Z_1^2 + \frac{\sum_{j=2}^p Z_j^2}{\lambda^{2}} + 1 \right)}{\lambda^{2}\sqrt{\left(Z_1^2 + \frac{\sum_{j=2}^p Z_j^2}{\lambda^{2}} + 1 \right)^2 - 4Z_1^2}} \right\} \\
&= \mathrm{E} \frac{-Z_1^2+1}{\sqrt{(Z_1^2+1)^2 - 4Z_1^2} }= \mathrm{E}\frac{1-Z_1^2}{\sqrt{(1-Z_1^2)^2}} = \mathrm{E} \{\mathrm{sgn}(1 - Z_1^2)\}, 
\end{split}
\end{equation*}
where the sign of the limit and hence, the relative location of the minimizer with respect to $\lambda^2$ as $\lambda$ is large depends on the distribution of $Z$. Recalling now that we parametrized the argument as $\sqrt{t} o_1$, the claim follows. Note that for a general $\sigma >0$ we obtain the limit $\mathrm{E} \{\mathrm{sgn}(1 - \sigma^2Z_1^2)\}$ above.

\end{proof}

\subsection{Proofs related to Section \ref{sec:sample}}\label{sec:proofs_3}

We start with the following auxiliary result.

\begin{lemma}\label{lem:quantile_region}
    Let $Z_1, \ldots, Z_n$ be a random sample from the distribution of a continuous random variable $Z$, and let $r_n$ be a sequence of random variables converging in probability to a constant $r\in\mathbb{R}$. Then.
    \begin{align*}
        \frac{1}{n} \sum_{i = 1}^n \mathbb{I}(Z_i \leq r_n) \rightarrow_P P( Z \leq r ),
    \end{align*}
    as $n \rightarrow \infty$.
    \begin{proof}
    Write first
     $$\frac{1}{n} \sum_{i = 1}^n \mathbb{I}(Z_i \leq r_n)=\frac{1}{n} \sum_{i = 1}^n \left\{ \mathbb{I}(Z_i \leq r_n) - \mathbb{I}(Z_i \leq r)+\mathbb{I}(Z_i \leq r) \right\},    $$ 
     where $\frac{1}{n} \sum_{i = 1}^n\mathbb{I}(Z_i \leq r)\rightarrow_P P( Z \leq r ). $ Thus, it suffices to show that
     $$\frac{1}{n} \sum_{i = 1}^n \left|\mathbb{I}(Z_i \leq r_n) - \mathbb{I}(Z_i \leq r)\right|\rightarrow_P 0.$$
     Fixing now $\epsilon > 0$, we have that
     \begin{equation*}
     \begin{split}
     &P\left(\frac{1}{n} \sum_{i = 1}^n \left|\mathbb{I}(Z_i \leq r_n) - \mathbb{I}(Z_i \leq r)\right| > \epsilon\right) \\
     \leq& P\left(\frac{1}{n} \sum_{i = 1}^n  \mathbb{I}(Z_i \in \mathcal{B}(r,|r_n-r|)) > \epsilon\right),
     \end{split}
     \end{equation*}
     where $\mathcal{B}(r,|r_n-r|)$ is the closed ball (interval) centered at $r$ of radius $|r_n-r|.$ Let then $n$ be large enough and $\xi$ be chosen such that $P\left(|r_n-r| > \xi \right) < \delta/2$ and $P(Z\in\mathcal{B}(r,\xi)) = \epsilon/2$. Then,
     \begin{equation*}
     \begin{split}
     &P\left(\frac{1}{n} \sum_{i = 1}^n  \mathbb{I}(Z_i \in \mathcal{B}(r,|r_n-r|)) > \epsilon\right)\\
     &\leq P(|r_n-r| > \xi) + P\left(\frac{1}{n} \sum_{i = 1}^n  \mathbb{I}(Z_i \in \mathcal{B}(r,\xi)) > \epsilon \right)\\
     & < \frac{\delta}{2} +P\left(\left|\frac{1}{n} \sum_{i=1}^n \mathbb{I}(Z_i \in \mathcal{B}(r,\xi)) - P(Z\in\mathcal{B}(r,\xi)) \right| + P(Z\in\mathcal{B}(r,\xi)) > \epsilon \right)\\
     & \leq \frac{\delta}{2} +P\left(\left|\frac{1}{n} \sum_{i=1}^n \mathbb{I}(Z_i \in \mathcal{B}(r,\xi))- P(Z\in\mathcal{B}(r,\xi)) \right|  > \frac{\epsilon}{2} \right)  < \frac{\delta}{2}+\frac{\delta}{2} = \delta,
     \end{split}
     \end{equation*}
     when $n$ is large enough, concluding the proof.
    \end{proof}
\end{lemma}

\begin{proof}[Proof of Theorem \ref{thm:thm7}]

Define first $\mathcal{K}=\mathcal{B}(H+1)\cap \{v \mid v' o_1 \geq 0\}$, where $\mathcal{B}(H+1)$ is a closed ball in $\mathbb{R}^p$ with the radius $H+1$ and $H$ is as in the proof of Theorem~\ref{theo:existence}. We first show that any sequence of minimizers, say $h_n$, of \eqref{eq:sample_objective_function} over $\mathcal{K}$ converges in probability to $\psi o_1$. For this, we begin by showing that $f_{P_n}$ converges to $f_P$ a.s. uniformly on $\mathcal{K}$. Observe that $\mathcal{K}$ is closed as an intersection of two closed subsets, and is bounded since it is a subset of a bounded set $B(H+1)$. Thus, $\mathcal{K}$ is compact. 

Denote further 
$$
g:\mathcal{K}\times \mathbb{R}^p \to \mathbb{R} ,\quad g(v ; X)=\| X - v \|\| X + v \|-\| X \|^2.
$$
It is straightforward to see that for all $v \in \mathcal{K}$, $X \mapsto g(v ; X)$ is measurable, since it is continuous in $X$, and for every $X$, $v \mapsto g(v; X)$ is continuous. As shown in the proof of Theorem \ref{theo:moments}, $|g(v, X)|\leq \| v \|^2$, implying that $g$ is dominated on $\mathcal{K}$ by an integrable function.
Finally, let $X_1, \dots, X_n$ be an i.i.d. sample from the distribution of $X$. Strong uniform law of large numbers (ULLN) now implies that a.s.  
$$
f_{P_n}(v)=\frac{1}{n}\sum_{i=1}^ng(v;X_i) \, ^\rightarrow_\rightarrow{}_n \, \mathrm{E}\{ g(v;X) \},
$$     
where $^\rightarrow_\rightarrow{}_n$ denotes uniform convergence in $v \in \mathcal{K}$.

The proof of Theorem \ref{theo:existence} implies that even though $f_P$ has entire $\mathbb{R}^p$ as a domain, the global minimizers of $f_{P}$ can be found in $\mathcal{B}(H+1)$. Moreover, Theorem~\ref{theo:main_1} states that, under ellipticity, for $\phi_1(S_p)>1/2$, $f_P$ is uniquely minimized at $\pm\psi o_1$, for some $\psi>0$. Therefore, the restriction $f_P{}_{|\mathcal{K}}:\mathcal{K}\to\mathbb{R}$ is uniquely minimized at $\psi o_1$. In continuation, we work solely with the restriction $f_P{}_{|\mathcal{K}}$, however, for the simplicity of the notation, we denote it $f_P$. 

For every $n$, $v\in\mathcal{K}\mapsto f_{P_n}(v)$ is a continuous function on a compact set, having therefore a global minimizer $h_n\in\mathcal{K}$; $f_{P_n}(h_n) = \inf_{v\in\mathcal{K}}f_{P_n}(v)$. Finally, as we have shown that also $f_{P_n}\,{}^\rightarrow_\rightarrow{}_n\, f_P,$ Theorem 5.7 in \cite{van2000asymptotic} (the argmax theorem) implies that $h_n \rightarrow_p \psi o_1$, thus proving the statement.

We next extend the result to any sequence $v_n$ of minimizers over the sets $\mathcal{K}_n := \mathcal{B}(H_n + 1/2)\cap \{v \mid v' o_1 \geq 0\}$, where $H_n$ is defined in Appendix \ref{sec:radius}. Since the distribution of $\| X \|$ is continuous, we have $Q_{0.75}(G_n) \rightarrow_p Q_{0.75}(G)$. Moreover, by Lemma \ref{lem:quantile_region}, we have $P_n \{ X \in \mathcal{B}(R_{n0}) \} \rightarrow_p P \{ X \in \mathcal{B}(R_{0}) \}$, the coefficients of the quadratic polynomial in the proof of Theorem \ref{theo:existence} converge to those of the polynomial in Appendix \ref{sec:radius}. Consequently the roots of the former converge to those of the latter and, by the continuous mapping theorem, we have $H_n \rightarrow_p H$. Next, we define the auxiliary sequence $v_n^*$ which takes the value $v_n$ when $v_n \in \mathcal{K}$ and the value $0$ when $v_n \not\in \mathcal{K}$. From $H_n \rightarrow_p H$ we obtain that  $P(\mathcal{K}_n \subset \mathcal{K}) \rightarrow 1$ and, consequently, $P(v_n \in \mathcal{K}) \rightarrow 1$ as $n \rightarrow \infty$. This then implies that $v^*_n - v_n = o_p(1)$, making it sufficient to show that $v_n^* \rightarrow_p h$ where $h=\psi o_1$ is the population minimizer.


To see this, we observe that the sequence $v^*_n$ resides in the compact set $\mathcal{K}$. Hence, the desired result follows from Theorem 5.7 in \cite{van2000asymptotic} once we show that $f_{P_n}(v^*_n) \leq f_{P_n}(h) + o_p(1)$. Consider then the following four different combinations:
\begin{enumerate}
    \item If $v_n \in \mathcal{K}$ and $h \in \mathcal{K}_n$, then $f_{P_n}(v_n^*) = f_{P_n}(v_n) \leq f_{P_n}(h)$.
    \item If $v_n \in \mathcal{K}$ and $h \not\in \mathcal{K}_n$, then $f_{P_n}(v_n^*) = f_{P_n}(v_n) \leq f_{P_n}(0) = \mathcal{O}_p(1)$.
    \item If $v_n \not\in \mathcal{K}$ and $h \in \mathcal{K}_n$, then $f_{P_n}(v_n^*) = f_{P_n}(0) = \mathcal{O}_p(1)$.
    \item If $v_n \not\in \mathcal{K}$ and $h \not\in \mathcal{K}_n$, then $f_{P_n}(v_n^*) = f_{P_n}(0) = \mathcal{O}_p(1)$.
\end{enumerate}
Define now the intersection event $\mathcal{A}_n := \{ v_n \in \mathcal{K} \} \cap \{ h \in \mathcal{K}_n \}.$ Since $h\in B(H)$ and $H_n \to_p H$, it satisfies $P(\mathcal{A}_n) \rightarrow 1$, letting us write
\begin{align*}
    f_{P_n}(v_n^*) \leq \mathbb{I}(\mathcal{A}_n) f_{P_n}(h) + \mathbb{I}(\mathcal{A}_n^C) \mathcal{O}_p(1).
\end{align*}
Finally, since $\mathbb{I}(\mathcal{A}_n^C) = 1 - \mathbb{I}(\mathcal{A}_n) = o_p(1)$ and $f_{P_n}(h) = \mathcal{O}_p(1)$, we obtain
\begin{align*}
    f_{P_n}(v_n^*) \leq f_{P_n}(h) +(\mathbb{I}(\mathcal{A}_n) - 1) f_{P_n}(h) + o_p(1) \leq f_{P_n}(h) + o_p(1)
\end{align*}
 proving the claim.

\end{proof}

\begin{proof}[Proof of Theorem \ref{thm:thm8}]
For the sake of simplicity of the notation, we show the result for the sequence of minimizers $v_n$ of $f_n$ for which $v_n'o_1\geq 0$. 

We first use Theorem 5.23 in \cite{van2000asymptotic} to obtain an asymptotic linearization of the quantity $\sqrt{n}(v_n-\psi o_1)$. The conditions of this theorem are trivially fulfilled apart from two: the Lipschitz condition and the existence of a second-order Taylor expansion around $v := \psi o_1$. The latter of these follows from Theorem \ref{theo:second_derivatives} in Appendix \ref{sec:taylor} and the corresponding Hessian matrix must be positive definite since Theorem \ref{theo:main_1} states that $\pm v$ are the unique minimizers of the population objective function. To see that also the former holds, we use the nuclear norm form of the objective function in Lemma \ref{lem:nuclear_norm} and the reverse triangle inequality to write
\begin{align*}
    \left| \| X X ' - v_1 v_1' \|_* - \| X X ' - v_2 v_2' \|_* \right| \leq \| v_1 v_1' - v_2 v_2' \|_* = \| v_1 + v_2 \| \| v_1 - v_2 \|,
\end{align*}
where $\| v_1 + v_2 \|$ is bounded when $v_1, v_2$ are in a given neighbourhood of $v$. Consequently also the Lipschitz condition in \cite[Theorem 5.23]{van2000asymptotic} is satisfied and we obtain
\begin{align*}
    \sqrt{n}(v_n-\psi o_1) = -M^{-1} \sqrt{n}\nabla f_{P_n}(v) + o_p(1),
\end{align*}
where $M\in\mathbb{R}^{p\times p}$ is a symmetric matrix of the form $M=M_1-M_2-M_3+M_4+M_4',$
for
\begin{align*}
M_1&=\mathrm{E}\left[\frac{\|v+X\|}{\|v-X\|}+\frac{\|v-X\|}{\|v+X\|}\right] I_p\\
M_2&=\mathrm{E}\left[\frac{\|v+X\|}{\|v-X\|^3}(v-X)(v-X)'\right]\\
M_3&=\mathrm{E}\left[\frac{\|v-X\|}{\|v+X\|^3}(v+X)(v+X)'\right]\\
M_4&=\mathrm{E}\left[\frac{1}{\|v+X\|\|v-X\|}(v+X)(v-X)'\right],
\end{align*}
see the proof of Theorem \ref{theo:second_derivatives}. From Lemma \ref{lemma:uniformintegrability} it follows that each of the expected values is well-defined. Moreover, $\nabla f_{P_n}(v)$ denotes the quantity
\begin{align*}
    \nabla f_{P_n}(v) = \frac{1}{n} \sum_{i = 1}^n \left\{ \frac{ \| v + X_i \| } { \| v - X_i \| } ( v - X_i ) + \frac{ \| v - X_i \| } { \| v + X_i \| } ( v + X_i ) \right\}.
\end{align*}

We next derive a simplified form for the matrix $M^{-1}$. Let $U$ be any orthogonal matrix for which $Uo_1=o_1$. The specific structure of the covariance of $X$, along with ellipticity, gives $UX\sim X$. Orthogonal invariance of the Euclidean norm now gives 
$$
U M_i U'=M_i,\quad i=1,2,3,4.
$$
Lemma 13 in \cite{radojicic2025unsupervised} now states that there exist constants $a_i,b_i\in\mathbb{R}$ such that 
$$
M_i = a_i o_1o_1' + b_i(I_p-o_1o_1'),\quad i=1,2,3,4,
$$
where 
$$
a_i=o_1'M_io_1,\quad b_i=o_2'M_io_2,
$$
noting that $a_1=b_1$. Moreover, as  $X\sim-X$, $a_2=a_3$ and $b_2=b_3$. This further gives that 
\begin{align*}
M&=a_1I_p + (-2a_2+2a_4)o_1o_1'+(-2b_2+2b_4)(I_p-o_1o_1')\\
&=(a_1-2a_2+2a_4)o_1o_1' + (a_1-2b_2+2b_4)(I_p-o_1o_1').
\end{align*}
It is straightforward to verify that for $\alpha,\beta\neq 0$, $(\alpha o_1o_1' +\beta (I_p-o_1o_1'))^{-1}=\alpha^{-1} o_1o_1' +\beta^{-1} (I_p-o_1o_1')$,
finally giving that  
$$
M^{-1}=(a_1-2a_2+2a_4)^{-1}o_1o_1' + (a_1-2b_2+2b_4)^{-1}(I_p-o_1o_1').
$$

We next investigate limiting distribution of $\sqrt{n}\nabla f_{P_n}(v)$. The CLT gives that 
$$
\sqrt{n} \{ \nabla f_{P_n}(v) - \mu_0 \} \to_d \mathcal{N}(0,\Sigma_0),
$$
where 
$$
\mu_0=\mathrm{E}\left[\frac{ \| v + X \| } { \| v - X \| } ( v - X ) + \frac{ \| v - X \| } { \| v + X \| } ( v + X )  \right],
$$
$$
\Sigma_0=\mathrm{Cov} \left( \frac{ \| v + X \| } { \| v - X \| } ( v - X ) + \frac{ \| v - X \| } { \| v + X \| } ( v + X )\right).
$$
Observing that $\mu_0$ equals the gradient of the population objective calculated at the minimizer and that the minimizer is not at the boundary of the domain, we conclude that $\mu_0=0$. This additionally shows that $\Sigma_0$ takes the form
\begin{align*}
& \mathrm{E} \left[ \frac{ \| v + X \|^2 } { \| v - X \|^2 } ( v - X )( v - X )' + \frac{ \| v - X \|^2 } { \| v + X \|^2 } ( v + X )( v + X )' + 
2vv'-2XX'\right]\\
=& 2\mathrm{E} \left[ \frac{ \| v + X \|^2 } { \| v - X \|^2 } ( v - X )( v - X )'\right] +
2vv'-2\mathrm{Cov}(X).
\end{align*}
The condition $\mathrm{E}\|X\|^2<\infty$ in the theorem statement ensures that $\mathrm{Cov}(X)$ and the expected value above exist as finite. Using then the same argumentation as for $M_i$ above, we conclude 
$$
\mathrm{E} \left[ \frac{ \| v + X \|^2 } { \| v - X \|^2 } ( v - X )( v - X )'\right]=c_1o_1o_1'+d_1(I_p-o_1o_1'),
$$
for 
$$
c_1=\mathrm{E}\left[ \frac{ \| v + X \|^2 } { \| v - X \|^2 } \{ o_1'( v - X ) \}^2\right],\quad d_1=\mathrm{E}\left[ \frac{ \| v + X \|^2 } { \| v - X \|^2 } \{ o_2'( v - X ) \}^2\right].
$$
Additionally, as $v=\psi o_1$, $\|v\|^2=\psi^2$, it holds that 
$$
\mathrm{Cov}(X)-vv'=\sigma_1o_1o_1' +\sigma_2(I_p-o_1o_1')-\psi^2o_1o_1'=(\sigma_1-\psi^2)o_1o_1' +\sigma_2(I_p-o_1o_1'),
$$
where $\sigma_1,\sigma_2$ are the leading and the remaining eigenvalue of $\mathrm{Cov}(X)$, respectively. This finally gives 
$$
\Sigma=2(c_1-\sigma_1+\psi^2)(a_1-2a_2+2a_4)^{-2} o_1o_1' + 2(a_1-2b_2+2b_4)^{-2}(d_1-\sigma_2) (I_p-o_1o_1'),
$$
concluding the proof. The coefficients in the statement of the theorem are thus
\begin{align*}
    q_1 &= 2(c_1-\sigma_1+\psi^2)(a_1-2a_2+2a_4)^{-2}, \\
    q_2 &= 2(a_1-2b_2+2b_4)^{-2}(d_1-\sigma_2).
\end{align*}
\end{proof}

\begin{proof}[Proof of Lemma \ref{lem:decomposition}]
    Slutsky's theorem yields
    \begin{align*}
        \sqrt{n} ( \| h_n \| - \| h \| ) =& \sqrt{n} \frac{\| h_n \|^2 - \| h \|^2}{ \| h_n \| + \| h \|} \\
        =& \frac{1}{2 \| h \|} \left\{ \sqrt{n} (h_n - h)'h_n + h' \sqrt{n} (h_n - h) \right\} + o_p(1) \\
        =& \frac{h'}{\| h \|}  \sqrt{n} (h_n - h) + o_p(1)
    \end{align*}
    Arguing and expanding similarly, we get for the directions that
    \begin{align*}
        \sqrt{n} \left( \frac{h_n}{\| h_n \|} - \frac{h}{\| h \|} \right) =& \frac{1}{\| h_n \|} \sqrt{n} (h_n - h) + \sqrt{n} \left( \frac{1}{\| h_n \|} - \frac{1}{\| h \|} \right) h \\
        =& \frac{1}{\| h \|} Q \sqrt{n} (h_n - h) + o_p(1).
    \end{align*}
    The result now follows using the equivariance properties of the multivariate normal distribution.
\end{proof}

\begin{proof}[Proof of Theorem \ref{theo:pca}]
    From the proof of Theorem A.1 in \cite{radojivcic2021large} we directly obtain that
    \begin{align*}
        & \sqrt{n} (s_n u_n - o_1) \\
        =& - \frac{1}{\mathrm{E}(Z_1^2)(\lambda^2 - 1)} (I_p - o_1 o_1') \sqrt{n} \left\{ \frac{1}{n}\sum_{i = 1}^n X_i X_i' - \mathrm{E}(Z_1^2) O \Lambda^2 O' \right\} o_1 + o_p(1) \\
        =& - \frac{1}{\mathrm{E}(Z_1^2)(\lambda^2 - 1)} O (I_p - e_1 e_1') \sqrt{n} \left\{ \frac{1}{n}\sum_{i = 1}^n \lambda Z_{i1} \cdot \Lambda Z_i - \mathrm{E}(Z_1^2) \lambda^2 e_1 \right\} + o_p(1).
    \end{align*}
    In order to apply the CLT, what remains is to compute the covariance matrix
\begin{align*}
    \mathrm{Cov} (\lambda Z_1 \cdot \Lambda Z) = \lambda^2 \Lambda \mathrm{E} (Z_1^2 Z Z') \Lambda - \{ \mathrm{E} (Z_1^2) \}^2 \lambda^4 e_1 e_1'.
\end{align*}
By the symmetry and permutation invariance of the spherical distribution $Z$, we have $\mathrm{E} (Z_1^2 Z Z') = \mathrm{diag}(\mathrm{E}(Z_1^4), \mathrm{E}(Z_1^2 Z_2^2), \ldots , \mathrm{E}(Z_1^2 Z_2^2))$. Hence, the limiting covariance matrix equals
\begin{align*}
     \frac{\lambda^2 \mathrm{E}(Z_1^2 Z_2^2)}{\{ \mathrm{E}(Z_1^2) \}^2(\lambda^2 - 1)^2} O (I_p - e_1 e_1') O'.
\end{align*}

\end{proof}

  
  



\subsection{Proofs related to Section \ref{sec:algorithm}}\label{sec:proofs_4}

We being with the properties of the so-called modified gradient.

\begin{lemma}\label{lemma:kuhn}
    Let $X$ be as in Theorem \ref{thm:thm7} and let $H_n$ and $v_n$ be as in the proof of Theorem \ref{thm:thm7}, such that $v_n\in B(H_n+1)\backslash \partial B(H_n+1) $, 
    where for the set $\mathcal{A}$,  $\partial \mathcal{A}$ denotes the boundary of $\mathcal{A}$. Let $\nabla f_{P_n}^*:B(H_n+1)
    \to \mathbb{R}^p$ be the modified gradient of $f_{P_n}$ defined as
    $$
    \nabla f_{P_n}^*(v) = \begin{cases}
        \nabla f_{P_n}(v),\quad v\notin\{\pm X_i \mid i=1,\dots,n\}\\
        \max\{\|\nabla f_k\|-2\|X_k\|,0\}\,\frac{\nabla f_k}{\|\nabla f_k\|},\quad v = X_k\\
        \max\{\|\nabla f_{-k}\|-2\|X_k\|,0\}\,\frac{\nabla f_{-k}}{\|\nabla f_{-k}\|},\quad v = -X_k,
    \end{cases}
    $$
    where 
    \begin{align*}
    \nabla f_{k} &= \sum_{\substack{i\in\{1,\dots, n, -1, \dots, -n\}\\i\neq\pm k}}\frac{\|X_i+X_k\|}{\|X_i-X_k\|}(X_i-X_k), \\
    \nabla f_{-k} &= \sum_{\substack{i\in\{1,\dots, n, -1, \dots, -n\}\\i\neq\pm k}}\frac{\|X_i-X_k\|}{\|X_i+X_k\|}(X_i+X_k),
    \end{align*}
    and we let $X_{-j}:=-X_j$, $j=1,\dots,n$. Then, the following holds:
    \begin{itemize}
        \item[(a)]  $\nabla f_{P_n}^*$ is a.s. two times continuously differentiable in some $\delta_n$-neighborhood of the population minimizer $h$. Moreover, in this set $\nabla(\nabla f_{P_n}^*)(v)= H_{P_n}(v)$, where $H_{P_n}$ denotes the Hessian of $f_{P_n}$.
        \item[(b)]  $\nabla f_{P_n}^*(v_n) = 0$.
    \end{itemize}
\end{lemma}
\begin{proof}[Proof of Lemma \ref{lemma:kuhn}]
    By the continuity of the distribution $P$, the $2n$ vectors $X_1, \ldots, X_n, -X_1, \ldots, -X_n$ are all distinct a.s. As we have $\pm X_i\neq h$ a.s., where $h$ is the population minimizer, this means that there exists a  open ball with radius $\delta_n$ around $h$ such that $\pm X_i$, $i=1,\dots,n$, do not reside there, further giving that $\nabla f_{P_n}^*$ coincides with $\nabla f_{P_n}$ there. Claim (a) now follows from the fact that $\nabla f_{P_n}$ is two times continuously differentiable everywhere, except at $\pm X_i$, $i=1,\dots,n$. To prove statement (b) we differentiate two cases: if $v_n\neq \pm X_i$, since it is not on the boundary of the domain, and $f_{P_n}$ is differentiable at $v_n$, $0=\nabla f_{P_n}(v_n)=\nabla f_{P_n}^*(v_n)$. 

    Let now $v_n=X_k$ for some $k\in\{1,\dots,n\}$, and consider the change from $X_k$ to $X_k - tu$, for some unit vector $u$. Let $t > 0$ and focus now on the function 
    \begin{align*}
    \begin{split}
        & t\mapsto f_{P_n}(X_k-tu) \\
        =& \frac{1}{n} t\|2 X_k-tu\| + \frac{1}{n} \sum_{
    \substack{i = 1 \\ i\neq k}}^n \|X_k-tu-X_i\|\|X_k-tu+X_i\|,
        \end{split}
    \end{align*}
    where we have ignored, without loss of generality, the term $(1/n) \sum_{i = 1}^n \| X_i \|^2$. In particular, for $X_k\neq 0$, 
       \begin{align*}
    & \lim_{t \downarrow 0} \left\{  \frac{f_{P_n}(X_k-tu) - f_{P_n}(X_k)}{t} \right\} \\
    =& \frac{2}{n}\|X_k\| - \frac{1}{n}\sum_{\substack{i = 1 \\ i\neq k}}^n \left(\frac{\langle X_k-X_i,u \rangle}{\|X_k-X_i\|}\|X_k+X_i\|+\frac{\langle X_k+X_i,u \rangle}{\|X_k+X_i\|}\|X_k-X_i\|\right)\\
    =& \frac{1}{n}\left(2\|X_k\| + \langle \nabla f_k,u\rangle\right).
    \end{align*}
    In case $X_k=0$, the directional derivative is similarly computed to equal $0 = 2 \|X_k\| + \langle \nabla f_k,u\rangle$. Therefore, the upper expression for the derivative holds for all $X_k \in \mathbb{R}^p$.
    
    If $X_k$ is a minimizer of $f_{P_n}$, then $2\|X_k\| + \langle \nabla f_k,u\rangle\geq 0$ for every direction $u$, thus in particular for the direction of the steepest descent, which is $u= -\nabla f_k/\| \nabla f_k\|$. Thus, 
    $2\|X_k\| - \|\nabla f_k\|\geq 0$, further giving $\nabla f_{P_n}^*(X_k)=0$. The case $v_n= -X_k$ for some $k\in\{1,\dots,n\}$ is proven analogously.
\end{proof}

\begin{proof}[Proof of Lemma \ref{alg::lemma1}]
For $u,v\in\mathbb{R}^p,\,v\neq \pm X_i$, $i=1,\dots,n$,  define the auxiliary function 
$$
h(u,v)=\frac{1}{n}\sum_{i=1}^n\frac{\|X_i-v\|\|X_i+u\|^2}{\|X_i+v\|}+\frac{1}{n}\sum_{i=1}^n\frac{\|X_i+v\|\|X_i-u\|^2}{\|X_i-v\|},
$$
and observe that for $v\neq \pm X_i$, $h(v,v)=2f_{P_n}(v)$. For fixed $v\neq \pm X_i$, define further $g_v(u):=h(u,v)$. We now show that for every $v\neq \pm X_i$, $u\mapsto g_v(u)$ is strictly convex. To simplify the notation, let $\alpha_i=\alpha_i(v)=\frac{\|X_i-v\|}{\|X_i+v\|}>0$. Then, 
$$
g_v(u)=\frac{1}{n}\sum_{i=1}^n\alpha_i\|X_i+u\|^2+\frac{1}{n}\sum_{i=1}^n\alpha_i^{-1}\|X_i-u\|^2, 
$$
with the gradient and the Hessian
\begin{align}
\nabla g_v(u)&=\frac{2}{n}\sum_{i=1}^n(\alpha_i-\alpha_i^{-1})X_i+\frac{2}{n}\sum_{i=1}^n(\alpha_i+\alpha_i^{-1})u,\label{eq:g_grad}\\
Hg_v(u)&=\frac{2}{n}\sum_{i=1}^n(\alpha_i+\alpha_i^{-1})I_p,\label{eq:g_hess}
\end{align}
respectively. Since for all $i=1,\dots,n$ we have $\alpha_i>0$, \eqref{eq:g_hess} implies that $g_v$ is strictly convex. Moreover, solving $\nabla g_v(u)=0$, \eqref{eq:g_grad} gives that 
$$
u=\frac{-1}{\sum_{i=1}^n(\alpha_i+\alpha_i^{-1})}\sum_{i=1}^n(\alpha_i-\alpha_i^{-1})X_i=T(v)
$$ 
is the unique minimum of $g_v(u)$. Therefore, for a fixed $v\neq \pm X_i$, $i=1\dots,n$, $T (v)=\mathrm{argmin}_{u\in\mathbb{R}^p}h(u,v).$ Thus, 
\begin{align}\label{eq:v_Tv_ineq1}
    h(T(v),v)\leq h(v,v)=2f_{P_n}(v), 
\end{align}
and the equality holds iff $v=T(v)$. We next show that $h(T(v),v)\geq 2f_{P_n}(T(v))$, where the equality holds iff $v=T(v)$. To see this, we write
\begin{align}\label{eq:v_Tv_ineq2}
    h(u,v)=&~\frac{1}{n}\sum_{i=1}^n\frac{(\|X_i-v\|\|X_i+u\|)^2+(\|X_i+v\|\|X_i-u\|)^2}{\|X_i-v\|\|X_i+v\|}\nonumber\\
    =&~\frac{1}{n}\sum_{i=1}^n\frac{\left(\|X_i-v\|\|X_i+u\|-\|X_i+v\|\|X_i-u\|\right)^2}{\|X_i-v\|\|X_i+v\|}\nonumber\\
    +&~\frac{2}{n}\sum_{i=1}^n\frac{\|X_i-v\|\|X_i+u\|\|X_i+v\|\|X_i-u\|}{\|X_i-v\|\|X_i+v\|}\nonumber\\
    =&~ \frac{1}{n}\sum_{i=1}^n\frac{\left(\|X_i-v\|\|X_i+u\|-\|X_i+v\|\|X_i-u\|\right)^2}{\|X_i-v\|\|X_i+v\|} + 2f_{P_n}(u)\nonumber\\
    \geq&~2f_{P_n}(u).
\end{align}
Equations \eqref{eq:v_Tv_ineq1} and \eqref{eq:v_Tv_ineq2} now give $2f_{P_n}(T(v))\leq h(T(v),v)\leq h(v,v)=2f_{P_n}(v)$, with the equality iff $T(v)=v$, showing the desired claim.



\end{proof}

\begin{proof}[Proof of Lemma \ref{alg::lemma2}]

Let $v_{\mathrm{old}}=X_k$ for some $k = 1,\dots,n$, and assume first that $\|\nabla f_k\|-2\|X_k\| > 0$. Consider now the change from $X_k$ to $X_k - td_k$, for some $t > 0$ and a unit vector $d_k = -\nabla f_k/\| \nabla f_k\|$, where $\nabla f_k$ is as defined in Lemma~\ref{lemma:kuhn}. Let $t > 0$ and focus now on the function 
    \begin{align*}
    \begin{split}
        & t\mapsto f_{P_n}(X_k-td_k) \\
        =& \frac{1}{n} t\|2 X_k-td_k\| + \frac{1}{n} \sum_{
    \substack{i = 1 \\ i\neq k}}^n \|X_k-td_k-X_i\|\|X_k-td_k+X_i\|,
        \end{split}
    \end{align*}
    where we again ignore, without loss of generality, the term $(1/n) \sum_{i = 1}^n \| X_i \|^2$. The proof of Lemma \ref{lemma:kuhn} gives that 
     \begin{align*}
     \lim_{m \to \infty} \left[  m\{f_{P_n}(X_k-\tfrac{1}{m}d_k) - f_{P_n}(X_k) \} \right] = \frac{1}{n}\left(2\|X_k\| - \|\nabla f_k\|\right)=:L_k<0.
    \end{align*}
    By the definition of limit, we can take $m$ large enough such that 
    $$
    |m \{ f_{P_n}(X_k-\tfrac{1}{m}d_k) - f_{P_n}(X_k) \} - L_k|<-L_{k}/2,
    $$
    further giving 
    $$
     f_{P_n} \left( X_k-\frac{1}{m}d_k \right) - f_{P_n}(X_k)<\frac{L_k}{2m}<0.
    $$
    For such a choice of $m$, we note that equality cannot be reached in $f_{P_n}(T_i(v))\leq f_{P_n}(v)$, showing that the claimed equivalence holds in case $\|\nabla f_k\|-2\|X_k\| > 0$. Consider then the situation $\|\nabla f_k\| - 2\|X_k\| \leq 0$, in which case, by definition, $T_k(X_k)=X_k$ and the claimed equivalence then holds trivially.
    
    Finally, the case $v_{old}= -X_k$ for some $k\in\{1,\dots,n\}$ is proven analogously.
\end{proof}

\begin{lemma}\label{lem:Tv_is_v_1}
    Under the conditions of Lemma \ref{alg::lemma1}, if $T(v) = -v$, then we must have $v = 0$.
\end{lemma}

\begin{proof}[Proof of Lemma \ref{lem:Tv_is_v_1}]
If $T(v) = -v$ then, by the definitions of $T(v)$ and $L(v)$, we have
\begin{align}\label{eq:Tv_is_v_1}
    \frac{1}{n} \sum_{i = 1}^n \left\{ \frac{ \| v + X_i \| } { \| v - X_i \| } ( v + X_i ) + \frac{ \| v - X_i \| } { \| v + X_i \| } ( v - X_i ) \right\} = 0.
\end{align}
After combining the two terms, the numerator of the $i$th summand reads
\begin{align*}
    2 \| v \|^2 v + 4 X_i X_i'v + 2 \| X_i \|^2 v.
\end{align*}
Multiplying \eqref{eq:Tv_is_v_1} by $v'$ from the left thus gives
\begin{align*}
    \frac{1}{n} \sum_{i = 1}^n \frac{ 2 \| v \|^4 + 4 (X_i'v)^2 + 2 \| X_i \|^2 \| v \|^2 } { \| v - X_i \| \| v + X_i \| } = 0,
\end{align*}
which implies that $v = 0$.
\end{proof}

We note that a simpler way to prove Lemma \ref{lem:Tv_is_v_1} is through Lemma \ref{alg::lemma1}, as follows. Assume that $T(v) = -v \neq 0$. Then, by Lemma \ref{alg::lemma1}, we have $f_{P_n}(v) = f_{P_n}(-v) = f_{P_n}(T(v)) < f_{P_n}(v)$, which is a contradiction.

\begin{lemma}\label{lem:Tv_is_v_2}
    Under the conditions of Lemma \ref{alg::lemma2}, if $T_i(v) = -v$, then we must have $v = 0$.
\end{lemma}

\begin{proof}[Proof of Lemma \ref{lem:Tv_is_v_2}]
    Assume that $v = X_i$ and $T_i(v) = - v$. If $\|\nabla f_i\|-2\|X_i\| \leq 0$, then the claim follows trivially. Whereas, if $\|\nabla f_i\|-2\|X_i\| > 0$, we have $2 X_i = - \varepsilon \nabla f_i/\|\nabla f_i\|$. This implies that $\varepsilon = 2 \| X_i \|$, which, together with the assumption that $\varepsilon < M$ shows that we must have $X_i = 0$ and $\varepsilon = 0$, leading to a contradiction. Hence, the claim is true.
\end{proof}

\begin{proof}[Proof of Corollary~\ref{cor:alg}]
    Fix $\Delta>0$, and let $v_1,v_2,\dots$ and $f_{P_n}(v_1),f_{P_n}(v_2),\dots$ be the sequences of optimizers obtained by Algorithm \ref{alg} and the corresponding functional values. 
    
    Lemmas \ref{alg::lemma1} and \ref{alg::lemma2} guarantee that $f_{P_n}(v_{k+1})\leq f_{P_n}(v_k)$ for every $k$, making the sequence fo the functional values decreasing. Moreover, $f_{P_n}(v)\geq 0$ for every $v\in\mathbb{R}^p$. Therefore, the sequence $(f_{P_n}(v_k))_k$ is bounded from below and thus convergent. 
    
    Denote by $L$ the corresponding limit. Convergence implies that there exists a finite number of steps $k_0\in\mathbb{N}$, such that for every $k\geq k_0$ $|f_{P_n}(v_k)-L|<\Delta/2$, which further gives
    \begin{equation*}
    |f_{P_n}(v_k) - f_{P_n}(v_{k+1})| \leq  |f_{P_n}(v_k) - L| + |f_{P_n}(v_{k+1})-L|<\Delta.
    \end{equation*}

    The above argument shows that the claimed convergence holds, as long as we do not hit \texttt{Step 12} where the error $err$ is artificially increased. Hence, we next still show that \texttt{Step 12} can not cause diverging behavior, in the sense of loops. First we note that \texttt{Step 12} can be reached at most $n$ times when running the algorithm. This is because if we are at \texttt{Step 12}, then $v_\mathrm{new} \neq v_\mathrm{old}$ and, by Lemma \ref{alg::lemma2}, we have the strict inequality $f_{P_n}(v_\mathrm{new}) < f_{P_n}(v_\mathrm{old})$. Consequently, since iteration of the algorithm cannot increase the objective function value (see Lemmas \ref{alg::lemma1} and \ref{alg::lemma2}), \texttt{Step 12} can occur at most once per sample point. Hence, when running the algorithm, either of two cases can happen: (1) The algorithm converges before hitting \texttt{Step 12} a total of $n$ times. (2) The algorithm hits \texttt{Step 12} a total of $n$ times, after which every new iteration satisfies either $v_\mathrm{new} = v_\mathrm{old}$ and $err = 0$ (convergence) or $f_{P_n}(v_\mathrm{new}) < f_{P_n}(v_\mathrm{old})$ (leading to eventual convergence since the sequence of objective function values is convergent).
\end{proof}

\section{Explicit forms for the integral \eqref{eq:durre_integral}}
\label{sec:integration}

Let $I_p(\lambda)$ denote the value of the integral \eqref{eq:durre_integral} for a given $p$. Then we have the following result.

\begin{theorem}
    We have,
    \begin{align*}
        I_2(\lambda) &= \frac{\lambda}{\lambda +1} \\
        I_3(\lambda) &= \frac{\lambda^2}{\lambda^2-1} \left(1- \frac{\arccos(\frac{1}{\lambda})}{\sqrt{\lambda^2-1}} \right) \\
        I_4(\lambda) &= \frac{\lambda^2}{(\lambda+1)^2}
    \end{align*}
    \begin{proof}
        Let us first consider the case $p=2$. The change of variable $u = (1+x)^{-\frac12}$ gives $x = 1/u^2-1$ and $dx = -2u^{-3} du$. Furthermore,
        \begin{equation*}
        \begin{split}
        \int (1+\lambda^2x)^{-\frac{3}{2}}(1+x)^{-\frac{1}{2}} dx &= -2\int \frac{u}{u^3\left(1+\lambda^2\left(\frac{1}{u^2} - 1\right)\right)^\frac32} du \\
        &= -2\int \frac{u}{\left(\lambda^2- (\lambda^2-1)u^2\right)^\frac32} du,
        \end{split}
        \end{equation*}
        where $s = \lambda^2- (\lambda^2-1)u^2$ gives $du = \frac{ds}{-2(\lambda^2-1)u}$ and the following antiderivative
        \begin{equation*}
        \begin{split}
        F_2(x) := \int (1+\lambda^2x)^{-\frac{3}{2}}(1+x)^{-\frac{1}{2}} dx &= \frac{1}{\lambda^2-1} \int s^{-\frac32} ds = - \frac{2}{\lambda^2-1} s^{-\frac12}\\
        &= - \frac{2}{\lambda^2-1} (\lambda^2 - (\lambda^2-1)u^2)^{-\frac12} = \\
        &= - \frac{2}{\lambda^2-1} (\lambda^2 - \frac{\lambda^2-1}{x+1})^{-\frac12}.
        \end{split}
        \end{equation*}
        Now,
        \begin{equation*}
        \begin{split}
        I_2(\lambda) &= \frac{\lambda^2}{2}(F_2(\infty) - F_2(0)) = \frac{\lambda^2}{2}\left(-\frac{2}{\lambda^2-1} \frac{1}{\lambda} + \frac{2}{\lambda^2-1}\right)\\
        &= \frac{\lambda^2}{\lambda^2-1}\left(1 - \frac{1}{\lambda}  \right) = \frac{\lambda}{\lambda +1}
        \end{split}
        \end{equation*}
For $p=3$ we apply the change of variable $u= (1+\lambda^2 x )^\frac12$ yielding $x = (u^2-1)/\lambda^2$ and $dx = 2 (1+\lambda^2x)^\frac12/\lambda^2,$ and
 \begin{equation*}
 \begin{split}
F_3(x) :=  \int (1+\lambda^2x)^{-\frac{3}{2}}(1+x)^{-\frac{1}{2}} dx = \frac{2}{\lambda^2} \int \frac{\lambda^2}{u^2(u^2+\lambda^2-1)} du, 
 \end{split}
 \end{equation*}
where the integrand can be decomposed as
$$\frac{1}{u^2(u^2+\lambda^2-1)}  = \frac{1}{(\lambda^2-1)u^2} - \frac{1}{(\lambda^2-1)^2 + (\lambda^2-1)u^2}.$$
Hence,
\begin{equation*}
    \begin{split}
        F_3(x) &= \frac{2}{\lambda^2-1} \int \frac{1}{u^2}- \frac{1}{(\lambda^2-1) + u^2} du \\
        &=   -\frac{2}{\lambda^2-1}u^{-1} - \frac{2}{(\lambda^2-1)^2}\int \frac{1}{(1 + \frac{u^2}{\lambda^2-1})} du . 
    \end{split}
\end{equation*}
For the integral above, we set $z = u/\sqrt{\lambda^2-1}$ giving $du = \sqrt{\lambda^2-1} dz$ and
\begin{align*}
 - \frac{2}{(\lambda^2-1)^2}\int \frac{1}{(1 + \frac{u^2}{\lambda^2-1})} du =& - \frac{2}{(\lambda^2-1)^\frac32}   \int \frac{1}{1 + z^2} dz \\
 =& - \frac{2}{(\lambda^2-1)^\frac32} \arctan(z).
\end{align*}
Consequently,
\begin{equation*}
    \begin{split}
F_3(x) &=  -\frac{2}{(\lambda^2-1)u} - \frac{2}{(\lambda^2-1)^\frac32} \arctan(\frac{u}{\sqrt{\lambda^2-1}}) \\
&= -\frac{2}{\lambda^2-1} \left(\frac{1}{\sqrt{\lambda^2 x +1}} + \frac{1}{\sqrt{\lambda^2-1}} \arctan\left(\frac{\sqrt{\lambda^2x+1}}{\sqrt{\lambda^2-1}}\right)\right).
    \end{split}
\end{equation*}
Hence,
\begin{equation*}
    \begin{split}
        I_3(\lambda) &= \frac{\lambda^2}{2}(F_3(\infty)- F_3(0))\\
        &= -\frac{\lambda^2}{\lambda^2-1}\left(\frac{\pi}{2\sqrt{\lambda^2-1}}  - 1  - \frac{1}{\sqrt{\lambda^2-1}} \arctan\left(\frac{1}{\sqrt{\lambda^2-1}}\right)\right)\\
        &= \frac{\lambda^2}{\lambda^2-1}\left( 1-  \frac{1}{\sqrt{\lambda^2-1}}\left( \frac{\pi}{2} - \arctan\left(\frac{1}{\sqrt{\lambda^2-1}}\right)   \right)\right)\\
        &= \frac{\lambda^2}{\lambda^2-1}\left( 1-  \frac{1}{\sqrt{\lambda^2-1}} \arctan\left({\sqrt{\lambda^2-1}}  \right)\right)
        \\&= \frac{\lambda^2}{\lambda^2-1}\left( 1-  \frac{1}{\sqrt{\lambda^2-1}} \arccos\left(\frac{1}{\lambda} \right)\right).
    \end{split}
\end{equation*}

Lastly, for $p=4$, let $u = 1/(1+x)^\frac 12$. Then $dx = -2(1+x)^\frac32 du$ and we obtain the antiderivative
\begin{equation*}
    \begin{split}
        & \int  (1+\lambda^2x)^{-\frac{3}{2}}(1+x)^{-\frac{3}{2}} dx \\
        =& -2\int \frac{1}{(1+\lambda^2(\frac{1}{u^2}-1))^\frac32} du\\
        =& -2\int \frac{1}{(1-\lambda^2+\frac{\lambda^2}{u^2})^\frac32} du\\
        =& -2 \frac{1}{(\lambda^2-1)^2} \left( \frac{\lambda^2}{\sqrt{\lambda^2 + (1-\lambda^2)u^2}} + \sqrt{\lambda^2 + (1-\lambda^2)u^2} \right)\\
        =:& -2G(u),
    \end{split}
\end{equation*}
since
\begin{equation*}
    \begin{split}
        G'(u) &= \frac{1}{(\lambda^2-1)^2}\left(-\frac{\lambda^2(1-\lambda^2)u}{(\lambda^2+(1-\lambda^2)u^2)^\frac32} + \frac{(1-\lambda^2)u}{(\lambda^2+(1-\lambda^2)u^2)^\frac12} \right)\\
        &= \frac{1}{\lambda^2-1}\left(\frac{\lambda^2u}{(\lambda^2+(1-\lambda^2)u^2)^\frac32} - \frac{u}{(\lambda^2+(1-\lambda^2)u^2)^\frac12} \right)\\
        &= \frac{1}{\lambda^2-1}  \frac{\lambda^2u - u(\lambda^2+(1-\lambda^2)u^2)}{(\lambda^2+(1-\lambda^2)u^2)^\frac32 } = \frac{u^3}{(\lambda^2+(1-\lambda^2)u^2)^\frac32}\\
        &= \frac{1}{(1-\lambda^2+\frac{\lambda^2}{u^2})^\frac32}.
    \end{split}
\end{equation*}
Finally
\begin{equation*}
    \begin{split}
        I_3(\lambda) &= - \lambda^2 (G(0) - G(1)) = -\frac{\lambda^2}{(\lambda^2-1)^2}(2\lambda - \lambda^2-1) = \frac{\lambda^2}{(\lambda^2-1)^2}(\lambda-1)^2\\
        &= \frac{\lambda^2}{(\lambda+1)^2(\lambda-1)^2}(\lambda-1)^2 = \frac{\lambda^2}{(\lambda+1)^2}.
    \end{split}
\end{equation*}
    \end{proof}
\end{theorem}

\section{On the norm of the minimizer}\label{sec:norm_minimizer}
In the case of $\phi_1(S_P) > 1/2$, the objective is to minimize
$$h_P(t) = \mathrm{E} \left\{\sqrt{(\Vert X \Vert^2 + t)^2 - 4t(o_1'X)^2} - \Vert X \Vert^2 \right\}$$
over $t > 0$. An integral presentation for the objective function is obtained via the density \eqref{eq:densityofX} of $X$. Due to the fact that the norm of the minimizer is a measure of scale, we may simply assume that $\sigma =1$.  Now, since $\Vert X \Vert^2 = \lambda^2 Z_1^2 + \sum_{j=2}^p Z_j^2$ and $X_1^2 = \lambda^2 Z_1^2$, we write
\begin{footnotesize}
\begin{equation*}
\begin{split}
& h_P(t) \\
=& \mathrm{E} \sqrt{(\Vert Z \Vert^2 +(\lambda^2-1)Z_1^2+ t)^2 - 4t\lambda^2Z_1^2}-\Vert Z \Vert^2 -(\lambda^2-1)Z_1^2\\
=& \mathrm{E} \sqrt{(\Vert Z \Vert^2 +(\lambda^2-1)\frac{Z_1^2}{\Vert Z \Vert^2}\Vert Z \Vert^2+ t)^2 - 4t\lambda^2\frac{Z_1^2}{\Vert Z \Vert^2}\Vert Z \Vert^2} -\Vert Z \Vert^2 -(\lambda^2-1)\frac{Z_1^2}{\Vert Z \Vert^2}\Vert Z \Vert^2\\
=& \mathrm{E} \left[ \mathrm{E} \left(\sqrt{(R^2 +(\lambda^2-1)R^2W+ t)^2 - 4t\lambda^2R^2W}-R^2-(\lambda^2-1)R^2W \Big| \Vert Z \Vert = R \right) \right],
\end{split}
\end{equation*}
\end{footnotesize}
where $W$ follows Beta $(\frac{1}{2}, \frac{p-1}{2})$ distribution. This gives
\begin{footnotesize}
\begin{equation*}
\begin{split}
    & h_P(t) \\
    =&  \frac{\Gamma(\frac{p}{2})}{\Gamma(\frac12)\Gamma(\frac{p-1}{2})}\mathrm{E} \left[ \int_0^1 \left(\sqrt{(R^2 +(\lambda^2-1)R^2w+ t)^2 - 4t\lambda^2R^2w}-R^2-(\lambda^2-1)R^2w\right) \right. \\
    & w^{-\frac12}(1-w)^{\frac{p-3}{2}} dw \biggr] \\
    =&  \frac{2\pi^{\frac{p}{2}}}{\Gamma(\frac12)\Gamma(\frac{p-1}{2})} \int_0^\infty r^{p-1} q(r) \int_0^1\left(\sqrt{(r^2 +(\lambda^2-1)r^2w+ t)^2 - 4t\lambda^2r^2w} \right. \\
    & -r^2-(\lambda^2-1)r^2w \biggl) w^{-\frac12}(1-w)^{\frac{p-3}{2}} dw dr,
\end{split}
\end{equation*}
\end{footnotesize}
where the density of $\Vert Z \Vert$ is of the form
$$f_{\Vert Z \Vert} (r) = \frac{2\pi^{\frac{p}{2}}}{\Gamma(\frac{p}{2})}r^{p-1}q(r) $$
with $f_Z(z) = q(\Vert z \Vert)$.

\section{Radius of the optimization domain}\label{sec:radius}

In Section \ref{sec:sample} we carry out the minimization over the set $\mathcal{B}(H_n + 1/2)$ defined as follows. Denote the empirical distribution of $\| X_1 \|, \ldots, \| X_n \|$ by $G_n$ and let $R_{n0} := Q_{0.75}(G_n)$.
Furthermore, define
\begin{align*}
    y_n(h) := [ 2 P_n \{ X \in \mathcal{B}(R_{n0}) \} - 1 ] h^2 - 2 h R_{n0} - R_{n0}^2 - f_{P_n}(0),
\end{align*}
and let then $H_n := H_{n0} + 1$ where $H_{n0}$ is the largest root of the polynomial $y_n(h)$. 
Arguing as in the proof of Theorem \ref{theo:existence}, all minimizers of $f_{P_n}$ then reside in the set $\mathcal{B}(H_n + 1/2)$.


\section{Second-order Taylor expansion of the objective function}\label{sec:taylor}

In the proofs of Lemma \ref{lemma:uniformintegrability} and Theorem \ref{theo:second_derivatives} the notation $C$ denotes arbitrary fixed constants that may change from line to line.

\begin{lemma}
    \label{lemma:uniformintegrability}
    Let $p>1$, and let $Z$ be a $p$-dimensional random vector with a bounded probability density function. Then, the families 
    $$\mathcal{C}_1 :=\left\{\frac{1}{\|Z-v\|}\bigg| v\in\mathbb{R}^p\right\}\quad\text{and}\quad \mathcal{C}_2:=\left\{\frac{1}{\|Z+v\|}\bigg| v\in\mathbb{R}^p\right\}$$
    of random variables are uniformly integrable. Consequently, they are bounded in $L^1.$
    \begin{proof}
    Let $A_{p-1}$ denote the surface area of the unit ball in $\mathbb{R}^p$ and let $f_Z$ be the probability density function of $Z$. Now,
    \begin{equation*}
        \begin{split}
        &\e \frac{1}{\|Z-v\|} \mathbb{I}\left(\frac{1}{\|Z-v\|} \geq K\right) = \int_{\mathbb{R}^p}\frac{1}{\|z-v\|} \mathbb{I}\left({\|z-v\|} \leq \frac{1}{K}\right)f_Z(z) dz \\
        &\leq \sup_{z\in\mathbb{R}^p}|f_Z(z)| \int_{\mathcal{B}(v, 1/K)}\frac{1}{\|z-v\|} dz \leq C \int_{\mathcal{B}(0,1/K)}\frac{1}{\|y\|} dy\\
        &= C \int_0^{1/K}A_{p-1}r^{p-1}\frac{1}{r} dr = \frac{CA_{p-1}}{p-1} \left(\frac{1}{K}\right)^{p-1}\to 0 
        \end{split}
    \end{equation*}
    as $K\to\infty$. This shows that the family $\mathcal{C}_1$ is uniformly integrable, and the family $\mathcal{C}_2$ can be treated similarly.
    \end{proof}
\end{lemma}

In Theorem \ref{theo:second_derivatives}, the notation $H f$ denotes the Hessian of the function $f$.

\begin{theorem}\label{theo:second_derivatives}
    Let $X\sim P\in \mathcal{E}(O, \Lambda)$ have a bounded probability density function. Then, the map $v \mapsto f_P(v) = \mathrm{E} \left( \| X - v \| \| X + v \| - \| X \|^2 \right)$ is twice continuously differentiable. In particular, let $v^*\in\mathbb{R}^p$, then $f_P(v)$ admits the second order expansion
    \begin{align*}
        & f_P(v) \\
        =& f_P(v^*) + \nabla f_P(v^*)' (v - v^*) + \frac{1}{2} (v - v^*)' Hf_P(v^*) (v - v^*) + o(\| v - v^* \|^2). 
    \end{align*}
    about $v^*$.
\end{theorem}

\begin{proof}[Proof of Theorem \ref{theo:second_derivatives}]
    Let $g(v) = \| X - v \| \| X + v \| - \| X \|^2,$
    $$h(v) = -\| X + v\| \frac{X-v}{\| X-v\|} + \|X-v\| \frac{X+v}{\| X+v\|}$$
    and
    \begin{equation*}
    \begin{split}
    l(v) &= \frac{\|X+v\|}{\|X-v\|} \left(I_p - \frac{1}{\| X-v\|^2} (X-v)(X-v)' \right)\\
    &+ \frac{\|X-v\|}{\|X+v\|} \left(I_p - \frac{1}{\| X+v\|^2} (X+v)(X+v)' \right)\\
    &- \frac{1}{\|X-v\|\|X+v\|}\left((X-v)(X+v)'+(X+v)(X-v)'\right).    
    \end{split}
    \end{equation*}
    Then $f_P(v) = \mathrm{E} g(v)$, and tedious computations verify that $\nabla g(v) = h(v)$ and $Hg(v) = l(v)$ almost surely.

    We start by showing that $\nabla f_P(v) =\mathrm{E}\nabla g(v) = \e h(v)$ exists and is continuous in $v$. Let $v^*\in\mathbb{R}^p$ be fixed and let $v^{(i)}=(v_1^*,\dots,v_{i-1}^*,v_i, v_{i+1}^*,\dots,v_p^*)$ with $v_i\to v_i^*$. Now,
    \begin{equation*}
     \frac{f_P(v^{(i)})- f_P(v^*)}{ v_i-v_i^*} = \e \left\{ \frac{g(v^{(i)})-g(v^*)}{v_i- v_i^*}\right\},   
    \end{equation*}
    where
    $$ \frac{g(v^{(i)})-g(v^*)}{ v_i- v_i^*} \to \frac{\partial}{\partial v_i} g(v^*)$$ 
    almost surely, and
    \begin{equation*}
    \begin{split}
      &|g(v)-g(v^*)| = \left| \|X-v\|\|X+v\|-\|X-v^*\|\|X+v^*\|\right|\\
      &\leq  \|X-v\| \left| \|X+v\| - \|X+v^*\| \right| + \|X+v^*\| \left| \|X-v\| - \|X-v^*\| \right|\\
      &\leq \|X-v\| \|v-v^*\|+\|X+v^*\|\|v-v^*\|,
      \end{split}
    \end{equation*}
    by the reverse triangle inequality. Hence,
    \begin{equation*}
    \label{uniformbound}
       \left| \frac{g(v^{(i)})-g(v^*)}{ v_i- v_i^*}\right| \leq 2\|X\| +\|v^{(i)}\|+\|v^*\|  \leq 2\|X\| + C 
    \end{equation*}
    in a neighborhood of $v^*$. As, by our assumptions, $\mathrm{E}\| X \| < \infty$, the dominated convergence theorem yields
    \begin{equation*}
      \lim_{v_i\to v_i^*} \frac{f_P(v^{(i)})- f_P(v^*)}{ v_i-v_i^*} = \e \frac{\partial}{\partial v_i} g(v^*), 
    \end{equation*}
    and furthermore $\nabla f_P(v^*) = \e \nabla g(v^*) = \e h(v^*).$ Continuity follows also from the dominated convergence theorem, since $h(v_n)\to h(v^*)$ almost surely as $v_n\to v^*$ and 
    $$\| h(v_n) \| \leq 2\|X\| + 2\|v_n\| \leq 2\|X\| + C$$
    in a neighborhood of $v^*$.

    Next, we show that $H f_P(v) = \e Hg(v) = \e l(v).$ Notice that the $i$th partial derivative of $f_P(v)$ is 
    $$\e h_i(v) = \e \left\{ -\| X + v\| \frac{X_i-v_i}{\| X-v\|} + \|X-v\| \frac{X_i+v_i}{\| X+v\|} \right\}.$$
    Let $v^*$ and $v^{(j)}$ be as before. Then
    \begin{equation*}
        \frac{\frac{\partial}{\partial i}f_P(v^{(j)}) - \frac{\partial}{\partial i}f_P(v^*) }{v_j-v_j^*} = \e \left\{ \frac{h_i(v^{(j)})-h_i(v^*)}{v_j- v_j^*}\right\},
    \end{equation*}
    where
    $$\frac{h_i(v^{(j)})-h_i(v^*)}{v_j- v_j^*} \to \frac{\partial}{\partial j} h_i(v^*)$$
    almost surely. Let $j=i$. Then
    \begin{equation*}
        \begin{split}
        &|h_i(v^{(i)})-h_i(v^*)| \leq\left| -\|X+v^{(i)} \| \frac{X_i-v_i}{\| X- v^{(i)} \|}    + \|X+v^* \| \frac{X_i-v_i^*}{\| X- v^{*}\|} \right|  \\
        &+ \left| \|X-v^{(i)} \| \frac{X_i+v_i}{\| X+ v^{(i)} \|}    - \|X-v^* \| \frac{X_i+v_i^*}{\| X+ v^{*}\|} \right| =: A + B,
        \end{split}
    \end{equation*}
    where, by applying the reverse triangle inequality,
    \begin{equation}
    \label{A}
        \begin{split}
        A &\leq |\| X+v^*\| - \| X+v^{(i)}\| | \frac{|X_i-v_i|}{\|X-v^{(i)}\|} + \|X+v^*\| \left| \frac{X_i-v_i^*}{\|X-v^*\|}- \frac{X_i-v_i}{\|X-v^{(i)}\| } \right|\\
        &\leq |v_i-v_i^*| + \|X+v^*\| \left| \frac{(X_i-v_i^*)\|X-v^{(i)}\| - (X_i-v_i)\|X-v^*\|}{\|X-v^*\|\|X-v^{(i)}\|} \right|
        \end{split}
    \end{equation}
    and 
    \begin{equation*}
        \begin{split}
    &|(X_i-v_i^*)\|X-v^{(i)}\| - (X_i-v_i)\|X-v^*\||\\ 
    &\leq \| X-v^{(i)}\| |(X_i-v_i^*)-(X_i-v_i)| + (X_i-v_i)|\|X-v^{(i)}\|-\|X-v^*\||\\
    &\leq \| X-v^{(i)}\| |v_i-v_i^*| + |X_i-v_i||v_i-v_i^*|.
        \end{split}
    \end{equation*}
    Now, \eqref{A} gives 
    \begin{equation}
    \label{Adivided}
    \begin{split}
     \frac{A}{|v_i-v_i^*|} &\leq 1 + \|X+v^*\| \frac{\|X-v^{(i)} \| + |X_i-v_i|}{\|X-v^*\|\|X-v^{(i)}\|} \leq 1 + 2\frac{\|X+v^*\|}{\|X-v^*\|}\\
     &\leq 1+2 \frac{\|X-v^*\|+2\|v^*\|}{\|X-v^*\|}= 3+ 4\frac{\|v^*\|}{\|X-v^*\|}.
     \end{split}
    \end{equation}
    The term $B$ can be bounded similarly giving
    $$\frac{B}{|v_i-v_i^*|} \leq 3+ 4\frac{\|v^*\|}{\|X+v^*\|}.$$
    This, together with \eqref{Adivided} implies
    $$\left|\frac{h_i(v^{(i)})-h_i(v^*)}{v_i-v_i^*}\right| \leq 6 + 4\|v^*\| \left(\frac{1}{\|X-v^*\|} + \frac{1}{\|X+v^*\|} \right),$$
    which is integrable by Lemma \ref{lemma:uniformintegrability}. Now, the dominated convergence theorem yields
    $$ \lim_{v_i\to v_i^*}\frac{\frac{\partial}{\partial i}f_P(v^{(i)}) - \frac{\partial}{\partial i}f_P(v^*) }{v_i-v_i^*}= \e \frac{\partial}{\partial i} h_i(v^*)= \e\frac{\partial^2}{\partial i^2} g(v^*) = \e l_{ii}(v^*). $$
    The mixed derivatives can be treated in a similar way resulting into $H f_P(v^*) = \e Hg(v^*)= \e l(v^*).$
    
    For continuity of $H f_P(v)$ at $v^*$, we first note that $l(v_n)\to l(v^*)$ almost surely and hence also in probability, as $v_n\to v^*$. We show that the elements of such sequence $l(v_n)$ are uniformly integrable implying that $l_{ij}(v_n)\to l_{ij}(v^*)$ in $L^1$, and further giving that $H f_P(v_n) = \e l(v_n) \to \e l(v^*)= H f_P(v^*).$ Let us consider the first term
    \begin{equation*}
    \begin{split}
    a(v_n) &:= \frac{\|X+v_n\|}{\|X-v_n\|} \left(I - \frac{1}{\| X-v_n\|^2} (X-v_n)(X-v_n)' \right)\\
    &\leq \left(1+2\frac{\|v_n\|}{\| X - v_n\|} \right) \left(I - \frac{1}{\| X-v_n\|^2} (X-v_n)(X-v_n)' \right)
    \end{split}
    \end{equation*}
    of the expression of $l(v_n)$. The sequence $\| v_n\| / \|X-v_n\|$ in the scalar factor is uniformly integrable by Lemma \ref{lemma:uniformintegrability}. Moreover, the matrix term is uniformly bounded over $n$. Hence, an element of $a(v_n)$ is of the form $B_n + B_nY_n,$ where $B_n$ is uniformly bounded and $Y_n$ is uniformly integrable. Consequently, $B_n$ and $B_nY_n$ are uniformly integrable, and thus is also $B_n + B_nY_n.$ 

    Next, a similar argument holds for the second term of $l(v_n)$, whereas the last term is uniformly bounded. Hence, an element of $l(v_n)$ can be expressed as a sum of uniformly integrable sequences, making itself uniformly integrable.
    
    
\end{proof}

\bibliographystyle{apalike}
\bibliography{references}

\end{document}